\documentclass[10pt]{amsart}

\usepackage{amsmath,enumitem}
\usepackage{amssymb}
\usepackage{amsthm}
\usepackage{array}
\usepackage[all]{xy}
\usepackage{color,xcolor}
\usepackage{fancyhdr}
\usepackage{euscript}
\usepackage{graphics,graphicx}
\usepackage{cancel}
\usepackage{fancybox}
\usepackage{verbatim}
\usepackage[margin=2cm]{geometry}

\usepackage[normalem]{ulem}

\usepackage{tikz,tikz-cd}
\usetikzlibrary{arrows,matrix}

\usepackage[colorlinks=false,pagebackref,hyperindex]{hyperref}

\DeclareFontFamily{OT1}{rsfs}{}
\DeclareFontShape{OT1}{rsfs}{n}{it}{<-> rsfs10}{}
\DeclareMathAlphabet{\mathscr}{OT1}{rsfs}{n}{it}

\newcommand{\degsup}{\operatorname{degsupp}}
\newcommand{\nilsup}{\operatorname{nilsupp}}
\newcommand{\mcl}{\mathcal{L}}

\newcommand{\PP}{\mathbf P}

\newcommand{\ZZ}{\mathbf Z}
\newcommand{\RR}{\mathbf R}
\newcommand{\QQ}{\mathbf Q}

\newcommand{\NN}{\mathbf N}

\newcommand{\id}{\operatorname{id}}

\newcommand{\Ann}{\operatorname {Ann}}

\def\Spec{\operatorname	{Spec}}

\newtheoremstyle%
{custom}%
{}
{}
{}
{}
{}
{.}
{ }
{\thmname{}
\thmnumber{}%
\thmnote{\bfseries #3}}%

\newtheoremstyle%
{Theorem}%
{}%
{}%
{\itshape}%
{}%
{}%
{.}%
{ }%
{\thmname{\bfseries #1}%
\thmnumber{\;\bfseries #2}%
\thmnote{\;(\bfseries #3)}}%

\theoremstyle{definition}
\newtheorem{thm}{Theorem}[section]
\newtheorem{cor}[thm]{Corollary}
\newtheorem{lem}[thm]{Lemma}
\theoremstyle{definition}
\newtheorem{dff}[thm]{Definition}
\newtheorem{xmp}[thm]{Example}
\newtheorem{rmk}[thm]{Remark}

\newtheorem*{run}{Running Assumptions}

\newtheorem{mainthm}{Theorem}

\def\fa{\mathfrak{a}}
\def\fb{\mathfrak{b}}
\def\fm{\mathfrak{m}}
\newcommand{\mfm}{\mathfrak{m}}
\def\fn{\mathfrak{n}}
\def\fp{\mathfrak{p}}
\def\fq{\mathfrak{q}}

\newcommand{\fM}{\mathfrak{M}}

\newcommand{\Arg}{\rule{1ex}{1pt}}

\newcommand{\hsl}{\operatorname{HSL}}
\newcommand{\fte}{\operatorname{Fte}}
\newcommand{\depth}{\operatorname{depth}}
\newcommand{\fdp}{\operatorname{F-depth}}
\newcommand{\gfdp}{\operatorname{gF-depth}}
\newcommand{\supp}{\operatorname{Supp}}
\newcommand{\fexp}{\operatorname{F-exp}}
\newcommand{\fbp}[1]{\left[p^{#1}\right]}
\newcommand{\cP}{\mathcal{P}}

\newcommand{\im}{\operatorname{im}}

\newcommand{\ux}{\underline{x}}

\newcommand{\ui}{\underline{i}}

\colorlet{DG}{green!50!black}
\colorlet{DB}{blue!50!black}

\def\LEM[#1]{\footnote{ {\color{DG} #1  }  }}
\def\KM[#1]{\footnote{ {\color{DB} #1  }  }}

\title{Generalized $F$-depth and graded nilpotent singularities} 
\author{Kyle Maddox and Lance Edward Miller}

\begin{document}

\begin{abstract}
We address explicit constructions of new variants of $F$-nilpotent singularities. In particular, we explore how (generalized) weakly $F$-nilpotent singularities behave under gluing, Segre products, Veronese subrings, and the formation of diagonal hypersurface algebras. From these results, explicit examples are produced and we provide bounds on their Frobenius test exponents. To accomplish these tasks, we introduce the {\it generalized $F$-depth} in analogy to Lyubeznik's $F$-depth. These depth-like invariants track (generalized) weakly $F$-nilpotent singularities in a similar fashion as (generalized) depth tracks (generalized) Cohen-Macaulay singularities. 
\end{abstract}

\maketitle


\section{Introduction}

The class of $F$-nilpotent singularities, and its variants, have gathered significant recent interest \cite{ST15, HQ18, Mad19, Quy19, PQ19, HQ21, KMPS}. Introduced by Srinivas-Takagi as well as implicitly by Blickle and Lyubzenik \cite{Bli01}, these singularities ask for the natural Frobenius actions on local cohomology of a local ring to be as nilpotent as possible. In the case of isolated singularities, they are closely tied to delicate information about Hodge filtrations of lifts to characteristic $0$, see \cite[Conj. $H_n$]{ST15}. In purely algebraic terms, $F$-nilpotent singularities insist that lower local cohomology directly have nilpotent Frobenius. This condition is impossible for top local cohomology and instead one asks that the Frobenius is nilpotent on the largest possible submodule, i.e., the tight closure of zero. These stand in some sense orthogonal to the class of $F$-injective singularities, where the Frobenius is asked to be injective on all local cohomology, and it is easily checked that rings which have $F$-rational singularities are equivalently both $F$-injective and $F$-nilpotent. 

\

Variants of $F$-nilpotent singularities have also been recently considered, namely {\bf weakly $F$-nilpotent singularities} which only require the nilpotent condition on the lower local cohomology modules. We consider weakly $F$-nilpotent singularities as analogs of Cohen-Macaulay singularities by treating nilpotent Frobenius actions as analogous to vanishing. Indeed, it is trivial to see that rings which are both $F$-injective and weakly $F$-nilpotent are Cohen-Macaulay. However, the added difficulty of working with nilpotence in place of vanishing introduces several technicalities, which we spell out through the course of this article. Analogous to how generalized Cohen-Macaulay singularities replace vanishing of lower local cohomology with a finite length condition, the first named author introduced the class of {\bf generalized weakly $F$-nilpotent} singularities \cite{Mad19} and have been shown to enjoy similar properties to both weakly $F$-nilpotent and generalized Cohen-Macaulay singularities.

\

Lyubeznik introduced a depth like invariant, {\bf $F$-depth}, which measures the smallest cohomological degree of non-nilpotence and for a local ring $R$. As Cohen-Macaulay rings are those with maximal depth, weakly $F$-nilpotent rings are those with maximal $F$-depth. In this article, we introduce and study an analogous depth-like invariant, the {\bf generalized $F$-depth}, which plays the same role for generalized weakly $F$-nilpotent singularities. The importance of these invariants is that theorems about (generalized) weakly $F$-nilpotent singularities are enhanced to theorems about calculating lower bound for (generalized) $F$-depth. Such lower bounds carry interesting geometric information, for example \cite[Cor. 4.6]{Lyu06} shows for completely local ring of dimension at least two with separably closed residue field, the punctured spectrum is formally geometrically connected if and only if its $F$-depth is at least two also.  

\

The study of nilpotent singularity types so far has suffered from a lack of explicit examples. To correct this critical deficit, we demonstrate that nilpotence properties persist along various important and natural constructions, specifically gluing, Segre products, and the formation of Veronese subrings. To aid in this, we focus on adapting our results to the appropriate notions for standard graded rings over a field. For technical simplicity of the introduction, we state theorems only in terms of (generalized) weakly $F$-nilpotent singularities, but stress that all theorems in the paper are in stronger forms explicating the behavior of (generalized) $F$-depth along such constructions.

\

We start with the gluing construction. Geometrically, this asks when a variety $X$ is a proper union $X = Y_1 \cup Y_2$ with each $Y_1, Y_2$, and $Y_1 \cap Y_2$ all having a prescribed singularity type, must $X$ also? If this is true, it is said that that singularity type {\it glues}. The class of $F$-injective singularities glues for local rings by \cite{Sch09} and many other related classes have been shown to glue. For local rings $(R,\fm)$, this question is phrased algebraically as asking for two ideals $\fa_1, \fa_2 \subset R$ are ideals with $\fa_1 \cap \fa_2 = 0$, if $R/\fa_1,R/\fa_2$, and $R/(\fa_1+\fa_2)$ all satisfy a singularity condition, does $R$? We address this question for (generalized) weakly $F$-nilpotent singularities as follows.

\begin{mainthm}(Corollary \ref{thm:GlueWeakFNil}) Suppose $(R,\fm)$ is a local ring of dimension $d \geq 2$ with ideals $\fa_1, \fa_2 \subset R$ such that $\fa_1 \cap \fa_2 = 0$. Assume $\dim R/\fa_1 = \dim R/\fa_2 = d$ and that $\dim R/(\fa_1 + \fa_2)\ge d-1$. If $R/\fa_1,R/\fa_2$, and $R/\fa_1 + \fa_2$ are weakly $F$-nilpotent, so is $R$. Further, if $R$ is equidimensional and $R/\fa_1$, $R/\fa_2$, and $R/\fb$ are generalized weakly $F$-nilpotent, so is $R$.
\end{mainthm}

The remainder of the article concerns standard graded rings over a field, where we explore how the canonical Frobenius action interacts with the graded structure of the graded local cohomology modules. Of crucial importance is whether the degree $0$ part of the local cohomology modules are nilpotent, and indeed we define an invariant, similar to the classical $a$-invariant of a graded ring, which tracks for us the largest degree where the local cohomology is not nilpotent -- i.e., $b_j(R) = \sup\{ n \in \ZZ \colon H_\fm^j(R) \textrm{ is not nilpotent in degree } n\}$, and we also set $b(R)=\min b_j(R)$. With these tools in hand, we explore nilpotence properties of two of the most common constructions among graded rings; Segre products and Veronese subrings. We show that the $b_j$-invariants defined above play a similar role in controlling the nilpotence properties for these constructions as the $a$-invariant does for controlling the Cohen-Macaulay property.

\begin{mainthm}(Theorems \ref{thm:fdepthSegre} and Corollary \ref{cor:wFnSegre}) Suppose $R$ and $S$ are standard graded rings over the same field and set $T = R \# S$ the Segre product. If $R$ and $S$ are weakly $F$-nilpotent, then $T$ is weakly $F$-nilpotent if and only if $b(R)=b(S)=\infty$, in which case $b(T)=\infty$ as well. If $R$ and $S$ are generalized weakly $F$-nilpotent, so is $T$. 
\end{mainthm}

Segre products of the form $E \times \PP^1$ where $E$ is an elliptic curve are studied in \cite{SW05}, and are known to be among the simplest normal non-Cohen-Macaulay singularities in dimension three. To give an application of the result above, we study nilpotence properties of a similar class where $E$ is replaced by a degree $d$ Fermat hypersurface. 

\begin{mainthm}(Theorem \ref{thm:FermatSegre})
Let $p>d$ and $n \geq 2$, and let $R = k[x_0,\ldots,x_n]/(x_0^d + \cdots + x_{n-1}^d - x_n^d)$, $S = k[u,v]$, and $T=R\# S$. When $p \equiv -1 \bmod d$, $T$ is weakly $F$-nilpotent.
\end{mainthm}

The $F$-rationality of Veronese subrings have been explored in detail \cite{Sin00}, making them also a natural class to study for (generalized) weakly $F$-nilpotent singularities. 

\begin{mainthm}(Theorems \ref{thm:F-depth and gF-depth for Veronese} and \ref{thm:VeronesePuncSpec}) Let $R$ be a standard graded ring. If $R$ is (generalized) weakly $F$-nilpotent, so is $R^{(n)}$ for all $n$. If $R$ is weakly $F$-nilpotent and $F$-nilpotent on the punctured spectrum, then the following are equivalent.
\begin{enumerate}[label=(\alph*)]
\item $R$ is $F$-nilpotent,
\item $b(R)=\infty$,
\item for all $n \in\NN$, $R^{(n)}$ is $F$-nilpotent,
\item and for a single $n \in \NN$, $R^{(n)}$ is $F$-nilpotent.
\end{enumerate}
\end{mainthm}

Finally, we apply both the Segre and Veronese constructions to address (generalized) weak $F$-nilpotent properties of certain diagonal bigraded algebras introduced in \cite{KSSW09}, where the authors explore when these rings have $F$-rational or $F$-regular singularities. We use these diagonal algebras to construct rings with prescribed levels of nilpotence in their local cohomology modules, see Theorems~\ref{thm:f-depth of diagonal subalgebras} and \ref{thm:diagonalSubalgebraHypersurface}.

\

One major motivation for the study of these variants of $F$-nilpotent singularities is that they enjoy uniform control on Frobenius closures of parameter ideals. In particular, for any ideal $J$ in a local ring, the {\bf Frobenius test exponent} $\fte J$ is the smallest $e$ for which $(J^F)^{[p^e]} = J^{[p^e]}$. In a sense, $\fte J$ bounds the complexity of $J^F$, and is extremely helpful in computing the ideal $J^F$. In particular, a uniform upper bound on $\fte J$ over all ideals would be effective in understanding the Frobenius closure of all ideals simultaneously, but Brenner \cite{Bre06} showed this is generally not possible even for nice rings of low dimension. In contrast, some of the major results of \cite{KS06,HKSY06,Quy19,Mad19} are aimed at providing uniform control on $\fte \fq$ over all parameter ideals $\fq$ simultaneously in the presence of a singularity types defined in terms of the local cohomology modules, e.g. Cohen-Macaulay rings in \cite{KS06} and weakly $F$-nilpotent rings in \cite{Quy19}. These bounds are given in terms of the Hartshorne-Speiser-Lyubeznik numbers of the ring. For each of the constructions explored in this paper, we calculate the Hartshorne-Speiser-Lyubeznik numbers in terms of the input data, allowing us to give upper bounds for their Frobenius test exponents, see \ref{thm:HSLGluing}, \ref{cor:Fte for Segre Products}, \ref{cor:fte* for gwfn Segre}, and \ref{thm:Fte for Veronese} for these results.


\

\noindent {\bf Acknowledgments:} We are extremely grateful to Ian Aberbach, Luis N\'{u}\~{n}ez-Betancourt, Alessandra Costantini, Hailong Dao, Jack Jeffries, Paolo Mantero, Vaibhav Pandey, Thomas Polstra, and Austyn Simpson for enlightening discussions from which the work here greatly benefited.

\section{Preliminaries}

This section reviews basic facts about local cohomology, modules with Frobenius actions, tight closure, and Frobenius test exponents. We utilize the following running assumptions. 

\begin{run}Throughout the paper, we assume all rings are noetherian, (graded\footnote{Graded in the sense of Section 5.}) local with unique (homogeneous) maximal ideal $\mfm$, of prime characteristic $p > 0$, and $F$-finite, i.e. the Frobenius map $F:R \rightarrow R$ is finite.  We denote by $R^\circ$ the complement of the minimal primes; $R^\circ = R\setminus \bigcup_{\fp \in \min(R)} \fp$. For $e \geq 1$ and $R$-module $M$, denote by $F_*^eM$ the restriction of scalars of $M$ along $F^e \colon R\rightarrow R$. We use the standard convention that $\sup \varnothing = -\infty$ and $\inf\varnothing=\infty$.\end{run}

\subsection{Modules with Frobenius actions: the local case}\label{sec:R[F]-modules}

The singularity types we study here are defined in terms of the map on the local cohomology modules of $R$ induced by the Frobenius endomorphism $F:R\rightarrow R$. This induced map is called a Frobenius action, and generally we state our results in the context of modules with Frobenius actions whenever possible. 	

\begin{dff}
For $R$-modules $M$ and $N$, an additive map $\rho:M\rightarrow N$ is \textbf{$p^e$-linear} for some $e \in \NN$ if $\rho(rm)=r^{p^e}\rho(m)$ for all $r\in R$, $m \in M$. A $p$-linear endomorphism of $M$ is called a \textbf{Frobenius action} on $M$, and we may write $(M,\rho)$ to emphasize the pair.
\end{dff}

\noindent A $p^e$-linear map $M\rightarrow N$ is equivalent to an $R$-linear map $M\rightarrow F^e_*N$. Further, it is well known that every Frobenius action is equivalent to a left module structure over a non-commutative ring $R[F]$, see for example \cite{EH08} for the construction. For convenience, we will refer to a left $R[F]$-module, equivalently an $R$-module with a Frobenius action, as an {\bf $R[F]$-module}, and we may say that $M$ is a finitely generated or artinian $R[F]$-module to mean that $M$ is finitely generated or artinian $R$-module which has a given Frobenius action. An $R$-linear map which is also left $R[F]$-linear will be said to \textbf{commute with Frobenius}. 

\begin{rmk}We tacitly use that left modules over a non-commutative ring still form an abelian category. For example, we will repeatedly use homological algebra in the category of $R[F]$-modules to obtain induced Frobenius actions on (co)homology of complexes and connecting maps which commute with the Frobenius actions on each module and are \textit{a fortiori} $R$-linear. \end{rmk}

We now explicate the primary example of $R[F]$-modules we consider throughout this paper. A more carefully detailed version of the example below is given in \cite[Reminder~2.1]{Sha06}

\begin{xmp}
The most important examples of $R[F]$-modules are arguably local cohomology modules. Specifically, for each ideal $J$ in a ring $R$, the Frobenius endomorphism $F: R \rightarrow R$ induces a natural Frobenius action on the local cohomology modules $F: H^j_J(R) \rightarrow H^j_J(R)$. Explicitly, if $J = (x_1,\ldots, x_t)$, then we identify the top local cohomology module with support in $J$ with the direct limit $H^t_J(R) = \varinjlim R/(x_1^j,\ldots,x_t^j)$. For $\xi = [z+(x_1^j,\ldots,x_t^j)]$ in $H^t_J(R)$, $F(\xi)=[z^p+(x_1^{jp},\ldots,x_t^{jp})]$. We reserve the symbol $F$ for this canonical action on $H^j_J(R)$ or for the Frobenius endomorphism on a ring.\end{xmp}

\noindent We will need the following basic but important concepts which arise in the study of $R[F]$-modules. 

\begin{dff}\label{dff:R[F]-module basics}
Let $(M,\rho)$ and $(M',\rho')$ be $R[F]$-modules. 
\begin{itemize}
\item An $R$-submodule $N\subset M$ is \textbf{$\rho$-stable} or \textbf{$\rho$-compatible} if $\rho(N)\subset N$. The $\rho$-stable submodules of $M$ are exactly the left $R[F]$-submodules of $M$.

\item The \textbf{orbit closure} of a $\rho$-stable submodule $N\subset M$ is the $\rho$-stable $R$-submodule: \[N^\rho_M = \{ m \in M \mid \rho^e(m) \in N \text{ for some } e \in \NN\},\] and $N$ is called \textbf{$\rho$-closed} if $N^\rho_M = N$. 

\item We say $M$ is \textbf{nilpotent} if $0^\rho_M = M$, i.e. for each $m \in M$ there is an $e \in \NN$ such that $\rho^e(m)=0$.

\item We say $M$ is \textbf{generalized\footnote{We use the word ``generalized" here to indicate a finite length version of the condition considered as for generalized Cohen-Macaulay rings.} nilpotent} if $M/0^\rho_M$ is a finite length $R$-module.
\end{itemize}
\end{dff} 

\begin{xmp}
Let $(M,\rho)$ be a $R[F]$-module. For $I\subset R$ an ideal, $\rho(IM) \subset I^{[p]} \rho(M) \subset IM$, so $IM$ is a $\rho$-stable submodule of $M$. Note that $\sqrt{I}M \subset (IM)^\rho_M$. 
\end{xmp}

We now outline some preparatory facts and definitions about nilpotent $R[F]$-modules which will be aimed at constructing rings where the local cohomology modules have a prescribed level of nilpotence. We note that checking nilpotence can be done at the level of groups. 

\begin{rmk}\label{rmk:forgetful fctr}
Define a category $\mathcal{C}$ with objects consisting of pairs $(G,\rho)$ where $G$ is an abelian group and $\rho \colon G\rightarrow G$ is a fixed abelian group endomorphism and morphisms between pairs $(G,\rho)$ and $(G',\rho')$ consisting of abelian group homomorphisms $f \colon (G,\rho)\rightarrow (G',\rho')$ such that $f\circ \rho = \rho' \circ f$. There is a forgetful functor $\mathcal{F}$ from the category of $R[F]$-modules $(M,\rho)$ to $\mathcal{C}$, where $\mathcal{F}(M,\rho)=(M,\rho)$ is viewed as an abelian group with fixed endomorphism $\rho$. As such $\rho^e \colon M\rightarrow M$ vanishes if and only if $\mathcal{F}(\rho)^e \colon \mathcal{F}(M)\rightarrow \mathcal{F}(M)$ vanishes, so it suffices to check nilpotence of $(M,\rho)$ by passing to $\mathcal{C}$.
\end{rmk}

When attempting to show an $R[F]$-module is nilpotent, we are further aided by a striking uniformity property exhibited by finitely generated or artinian $R$-modules with Frobenius actions. 

\begin{dff}
Let $(M,\rho)$ be an $R[F]$-module. The \textbf{Hartshorne-Speiser-Lyubeznik number of $M$} is defined as follows: \[ \hsl M=\inf\{ e \in \NN \mid  \rho^e(m)=0 \text{ for all } m \in 0^\rho_M\} \in \NN \cup \{\infty\}. \] 
\end{dff}

Note if $M$ is finitely generated, then $\hsl M<\infty$ since $0^\rho_M$ is finitely generated. Another extremely important case where HSL numbers are finite is given below, due to Hartshorne-Speiser \cite[Prop.~1.11]{HS77}, Lyubeznik \cite[Prop.~4.4]{Lyu97}, and Sharp \cite[Cor.~1.8]{Sha06}. 

\begin{thm}[The Hartshorne-Speiser-Lyubeznik Theorem]\label{thm:HSL thm}
If $M$ is an artinian $R[F]$-module, then $\hsl M<\infty$.
\end{thm}

\begin{xmp}
For $(M,\rho)=(R,F)$, the orbit closure of $0$ in $R$ is $0^F_R=\sqrt{0}$, so under the definition above we have: \[\hsl R=\min\left\lbrace e \in \NN \mid \sqrt{0}^{\fbp{e}} = 0\right\rbrace.\] 
\end{xmp}

It is common to use the following alternate definition of $\hsl R$.

\begin{dff} The \textbf{Hartshorne-Speiser-Lyubeznik number of $R$} is defined $\hsl R=\max\{\hsl H^j_\fm(R)\mid 0\le j\le  d \}\in \NN$. Note under this definition, $\hsl R<\infty$ by Theorem~\ref{thm:HSL thm}.
\end{dff}

\noindent The utility of a single uniform exponent of Frobenius which annihilates all the nilpotent elements in every local cohomology module of $R$ cannot be overstated. 

\begin{lem}\label{lem:NilSES}
Suppose the following is a complex of $R[F]$-modules. 
\begin{center}
\begin{tikzcd}
0 \arrow{r} & A \arrow{r} & B \arrow{r}{\pi} & C \arrow{r} & 0
\end{tikzcd}
\end{center} If the sequence is exact, then $B$ is nilpotent if and only if $A$ and $C$ are nilpotent. Further, if each of $A$, $B$, and $C$ has finite Hartshorne-Speiser-Lyubeznik number, then $B$ is generalized nilpotent if and only if $A$ and $C$ are generalized nilpotent. 
\end{lem}

\begin{proof}
We prove only the second claim, since the first is a simpler version. Suppose $B$ is generalized nilpotent. As $\mfm^n A\subset \mfm^n B = 0^\rho_B$, its clear that $\mfm^n A$ is nilpotent. Further, $\pi$ is surjective so $\mfm^n C = \pi(\mfm^n B) = \pi(0^\rho_B)$, so $\mfm^n C \subset 0^\rho_C$. If instead $A$ and $C$ are generalized nilpotent, we can pick an $n \gg 0$ so that $\mfm^n A$ and $\mfm^n C$ are nilpotent. Thus $\pi(\mfm^n B )=\mfm^n\pi(B)=\mfm^nC=0^\rho_C$. 

If $\hsl C=e_0$, then $\rho^{e_0}\circ \pi=\pi\circ\rho^{e_0}$ kills $\mfm^n B$, i.e. $\rho^{e_0}(\mfm^n B)\subset A$. Since $\mfm^n A$ is nilpotent, the abelian subgroup generated by $F^{e_0}(\mfm^n) A$ is inside $0^\rho_A$, and we have \[
F^{e_0}(\mfm^n) \rho^{e_0}(\mfm^n B) = \rho^{e_0}(\mfm^{2n} B) \subset 0^\rho_A
\] so $\mfm^{2n} B$ is nilpotent as $0^\rho_A$ is $\rho$-closed.
\end{proof}

Lemma \ref{lem:NilSES} shows the full subcategory of nilpotent $R[F]$-modules form a (weak) Serre subcategory of $R[F]$-modules. Generally, if $\alpha:(M,\rho)\rightarrow (M',\rho')$ is $R[F]$-linear, we have $\alpha(0^\rho_M)\subset 0^{\rho'}_{M'}$, so $\alpha^{-1}(0^{\rho'}_{M'})\supset 0^\rho_M$. As in the main theorem of \cite{Mad19}, when lifts of nilpotent element are nilpotent, we can obtain bounds on the HSL number of $M'$ in terms of the HSL numbers for $\im(\alpha)$ and $\operatorname{coker}(\alpha)$. We codify this formally in the next lemma. 

\begin{lem}\label{lem:lift of nilp is nilp hsl bounds}
Suppose the following is a complex of $R[F]$-modules. \begin{center}
\begin{tikzcd}
A \arrow{r}{\alpha} & B \arrow{r}{\beta} & C
\end{tikzcd}
\end{center}
If the sequence is exact, then $\alpha^{-1}(0^\rho_B)= 0^\rho_A$ if and only if $\ker(\alpha)$ is a nilpotent $R[F]$-submodule of $A$, and if so, $\hsl B\le \hsl A+\hsl C$. Finally, if the sequence is split exact, then $\hsl B=\max\{\hsl A,\hsl C\}$.
\end{lem}

\begin{proof}
Since $\alpha$ commutes with Frobenius, we automatically have $\alpha(0^{\rho_A}_A)\subset 0^{\rho_B}_B$. Then $\alpha^{-1}(0^\rho_B) = \ker(\alpha)^\rho_A$. If $\ker(\alpha)$ is nilpotent, then $\ker(\alpha)\subset 0^\rho_A$, so $\ker(\alpha)^\rho_A = 0^\rho_A$ as $0^\rho_A$ is $\rho$-closed. If $0^\rho_A = \ker(\alpha)^\rho_A$, then $\ker(\alpha)\subset \ker(\alpha)^\rho_A= 0^\rho_A$, so $\ker(\alpha)$ is nilpotent.

The claimed bound on $\hsl B$ is trivial if either of $\hsl A$ or $\hsl C$ are infinite, so suppose $\hsl A=a$ and $\hsl C=c$ are finite. In this case $\beta(0^\rho_B)\subset 0^\rho_C$, so $\rho^c(0^\rho_B)\subset \ker(\beta)=\im(\alpha)$. Since $\im(\alpha)$ is a subquotient of $A$, $\hsl \im(\alpha) \le a$ and by the above, $\alpha^{-1}(\rho^c(0^\rho_B)) = 0^\rho_A$, i.e. $\rho^{c+a}(0^\rho_B)=0$. 

The final assertion in the case that the sequence is a split short exact sequence is immediate since its trivial to check that $0^{\rho_A\oplus \rho_C}_{B} = 0^{\rho_A}_A \oplus 0^{\rho_C}_C$. \end{proof}

\subsection{Tight and Frobenius closure of submodules} 

The Frobenius action on the top local cohomology of a Cohen-Macaulay local ring has been studied frequently for its connections to tight closure theory. K. E. Smith showed that in many cases, the tight closure of $0$ in $H^d_\fm(R)$ was the largest $F$-stable submodule of $H^d_\fm(R)$. See \cite[Prop.~2.5]{Smi97} and \cite[Discussion~2.10]{EH08} for more. We now review tight and Frobenius closures of modules. 

\begin{dff}\label{dff:TC}
Let $N\subset M$ be $R$-modules. For any $c \in R$ and $e \in \NN$, define the map $\mu_c^e \colon M \rightarrow (M/N)\otimes_R F^e_*R$ to be the composition of the maps below. \begin{center}
\begin{tikzcd}
M \arrow{r}{\pi} & M/N \arrow{r}{\id \otimes F^e} & (M/N)\otimes_R F^e_*R \arrow{rr}{\id \otimes  F^e_*(c)} & & (M/N)\otimes_R F^e_*R
\end{tikzcd}
\end{center} The \textbf{tight closure of $N$ in $M$}, denoted $N^*_M$, is defined as \[
N^*_M = \left\lbrace m \in M \mid \,\text{there is } c \in R^\circ \text{ such that } m \in \ker(\mu_c^e) \text{ for all } e \gg 0\right\rbrace .
\] and the \textbf{Frobenius closure of $N$ in $M$}, denoted $N^F_M$, is similarly defined as: \[ N^F_M = \left\lbrace m \in M \mid m \in \ker(\mu^e_1) \text{ for all } e \gg 0\right\rbrace .\] 
\end{dff}

By definition, we have $N \subset N^F_M \subset N^*_M \subset M$. If $M=R$ and $N=I$ is an ideal of $R$, $I^*_R$ and $I^F_R$ are usually written in a different, but equivalent manner as below.\begin{align*}
I^* & = \left\lbrace x \in R \mid \,\text{there is }c \in R^\circ \text{ such that }cx^{p^e}\in I^{\fbp{e}} \text{ for all } e \gg 0 \right\rbrace \\
I^F & = \left\lbrace x \in R \mid \, x^{p^e}\in I^{\fbp{e}} \text{ for some } e \in \NN \right\rbrace 
\end{align*}

\begin{rmk} 
Since the name \textit{Frobenius closure} is established in its classical use above, we are careful to use \textit{orbit closure} when a Frobenius action is involved. For example, given an ideal $I\subset R$, the orbit closure of $I$ inside the $R[F]$-module $(R,F)$ is simply the radical $\sqrt{I}$, typically much larger than its classical Frobenius closure $I^F$.

If $I=(x_1,\cdots,x_t)$, the canonical Frobenius map induced by Frobenius $F^e:H^t_I(R)\rightarrow F^e_*H^t_I(R)$ agrees with the map $\mu^e_1: H^t_I(R) \rightarrow H^t_I(R)\otimes_R F^e_*R$, but for the lower local cohomology modules the two are generally distinct. This implies that on a top local cohomology module, we need not distinguish between the Frobenius closure in Definition~\ref{dff:TC} and the orbit closure defined in Definition~\ref{dff:R[F]-module basics}.  
\end{rmk}

\begin{rmk}\label{rmk:testelements} The existential quantifier in Definition~\ref{dff:TC} can be simplified by using test elements. Recall, reduced excellent local rings $(R,\fm)$ admit big test elements, i.e. elements $c \in R^{\circ}$ so that for any pair of $R$-modules $N \subset M$, $\eta \in N_M^*$ if and only if $\mu_c^e(\eta) = 0$ for $e \gg 0$. Here big refers to lack of finite generation assumptions on the modules, and the nuance is that $c$ is independent of the modules $N$ and $M$. For proof, see \cite[Thm.~6.17, Cor.~6.26]{HH90}. 
\end{rmk}

\subsection{Frobenius test exponents}

The Frobenius test exponent was introduced by Katzman-Sharp in \cite{KS06}, where they show that $\hsl R$ can uniformly control the Frobenius closure of parameter ideals in a Cohen-Macaulay ring.  We direct readers to Huong-Quy \cite{HQ18} for a more details concerning Frobenius test exponent problem, but provide a short review below. Recall a \textit{parameter ideal} is one whose height is equal to its minimal number of generators.

\begin{dff}
Let $J\subset R$ be an ideal. The \textbf{Frobenius test exponent of $J$} is  defined as \[
\fte J=\min\{e\in \NN \mid (J^F)^{\fbp{e}}=J^{\fbp{e}}\}\in\NN ,
\] and the \textbf{Frobenius test exponent of $R$} is defined as\[
\fte R=\sup\{\fte \fq \mid \fq\subset R \text{ a parameter ideal}\}\in\NN\cup\{\infty\}.
\] 
\end{dff}

\noindent Note that $\fte J<\infty$ for all $J$ since $J^F$ is finitely generated. However, $\fte R$ is not guaranteed to be finite due to the possibility of a sequence of parameter ideals whose Frobenius test exponents tend to infinity \cite{Bre06}. When the ring has a nilpotent singularity type, HSL numbers are known to provide a uniform upper bound on $\fte \fq$ independent of the parameter ideal $\fq$ due to strong connections between local cohomology modules and parameter ideals, see Remark~\ref{rmk:finiteness theorems}. 

Few other classes of ideals are known to have uniformly upper bounded Frobenius test exponents. In a recent paper \cite{HQ21} of Huong-Quy, the authors give a sufficient condition in terms of generalized nilpotence to ensure that filter regular sequences of length at most a fixed value have uniformly bounded Frobenius test exponents, generalizing the main result of \cite{Mad19}.

\section{Nilpotent singularity types and depth-like invariants}

We remind the reader of a classical characteristic-free singularity type which generalizes the Cohen-Macaulay property. In particular, $R$ is \textbf{generalized Cohen-Macaulay} if $H^j_\fm(R)$ is finite length for $0 \le j < \dim R$. We now come to the definition of the nilpotent singularity types we will consider throughout the paper. 

\begin{dff}\label{dff:nilp singularity types}
We say $R$ is \textbf{weakly $F$-nilpotent} if $H^j_\fm(R)$ is nilpotent for $j < d$ and \textbf{generalized weakly $F$-nilpotent} if $H^j_\fm(R)$ is generalized nilpotent for all $j<d$. Finally, $R$ is \textbf{$F$-nilpotent} if it is weakly $F$-nilpotent and $0^*_{H^d_\fm(R)}=0^F_{H^d_\fm(R)}$.
\end{dff}

\noindent The names weakly and generalized weakly $F$-nilpotent were first used in \cite{Mad19}, though weakly $F$-nilpotent rings were studied earlier in \cite{Quy19} and \cite{PQ19}. Further, $F$-nilpotence was introduced in \cite{ST15}, though the notion appears earlier \cite[Def. 4.1]{BB05}. For readers more familiar with classical $F$-singularities, we note that in this setting, $R$ is both $F$-injective and $F$-nilpotent if and only if $R$ is $F$-rational, see \cite[Prop.~2.4]{ST15}. 

\

Our guiding intuition is that nilpotence is a robust replacement for vanishing; an idea arising in work of Lyubeznik. Under this perspective, it is natural to expect weakly $F$-nilpotent rings to share properties in common with Cohen-Macaulay rings and generalized weakly $F$-nilpotent rings to share properties in common with generalized Cohen-Macaulay rings. We make this analogy even more precise with the introduction with depth-like invariants in the later case. As an example of this perspective, we summarize recent work showing that nilpotent singularities exhibit finite Frobenius test exponents, a fact which was known first for Cohen-Macaulay rings due to Katzman-Sharp.

\begin{rmk} \label{rmk:finiteness theorems}
Let $h_j=\hsl H^j_\fm(R)$.
\begin{itemize}
\item If $R$ is Cohen-Macaulay, then $\fte R=\hsl R=h_d$. \cite[Thm.~2.4]{KS06}
\item If $R$ is generalized Cohen-Macaulay, then $\fte R<\infty$. \cite{HKSY06}.
\item If $R$ is weakly $F$-nilpotent, then $\fte R\le \sum_{j=0}^d {\binom{d}{j}} h_j$. \cite[Main Thm.]{Quy19} 
\item Suppose $R$ is generalized weakly $F$-nilpotent. Let $N\in \NN$ be the smallest $n$ so that $\fm^n H^j_\fm(R)\subset 0^F_{H^j_\fm(R)}$ for each $0 \le j < d$, finite by hypothesis. Further, let $e_1$ be the minimum $e \in \NN$ such that $p^e \ge 2^{d-1}N$. We then have $\fte R \le e_1+\sum_{j=0}^d \binom{d}{j} h_j$. \cite[Thm.~3.6]{Mad19}
\end{itemize}
\end{rmk}

\noindent This analogy between vanishing and nilpotence also informs a hierarchy among singularity types as outlined below. We note each vertical arrow is an equivalence if $R$ is $F$-injective.

\begin{center}\begin{tikzcd}
\text{{\em F}-rational} \arrow[Rightarrow]{r} \arrow[Rightarrow]{d} & \text{Cohen-Macaulay} \arrow[Rightarrow]{r} \arrow[Rightarrow]{d} & \text{generalized Cohen-Macaulay} \arrow[Rightarrow]{d} \\
\text{{\em F}-nilpotent} \arrow[Rightarrow]{r} & \text{weakly {\em F}-nilpotent} \arrow[Rightarrow]{r} & \text{generalized weakly {\em F}-nilpotent}
\end{tikzcd}\end{center}

\subsection{F-depth and generalized F-depth} \label{sec:(g)F-depth} 

Just as the classical theory of depth measures the first cohomological index with a nonvanishing local cohomology module, Lyubeznik's $F$-depth and the generalized $F$-depth we define here measure the first cohomological index $j$ where $H^j_\fm(R)$ is not (generalized) nilpotent. By demonstrating that depth and (generalized) $F$-depth act analogously in many situations, we are able to recapture many classic results about Cohen-Macaulay rings for rings with nilpotent singularity types. We typically accomplish this by providing lower bounds for (generalized) $F$-depth and then recalling that (generalized) weakly $F$-nilpotent rings as those for which the (generalized) $F$-depth is maximal. 

Note, since $F_*$ commutes with localization, an $R$-linear map $\rho_M: M \rightarrow F_* M$ induces an $R$-linear map $\rho \colon H^j_J(M)\rightarrow F_* H^j_J(M)$ for all ideals $J\subset R$. Equivalently, if $(M,\rho)$ is an $R[F]$-module, for any $J\subset R$ an ideal and $j\in\NN$ there is a natural map $\rho: H^j_J(M)\rightarrow H^j_J(M)$ induced by $\rho:M\rightarrow M$ so that $(H^j_J(M),\rho)$ is an $R[F]$-module. The construction of this induced map in the case of $(M,\rho)=(R,F)$ is spelled out explicitly in \cite[Reminder~2.1]{Sha06}.

\begin{dff}
The \textbf{$F$-depth} of an $R[F]$-module $(M,\rho)$ is \[ \fdp M=\inf\{j\in\NN \mid (H^j_\fm(M),\rho) \text{ is not nilpotent}\}.\] We define the \textbf{generalized $F$-depth} of $M$ to be \[
\gfdp M=\inf\{j\in\NN\mid (H^j_\fm(M),\rho) \text{ is not generalized nilpotent}\}. 
\] The $F$-depth and generalized $F$-depth of $R$ are understood to be with respect to the $R[F]$-module structure $(R,F)$.
\end{dff}

In \cite{Lyu06}, Lyubeznik used the $F$-depth of $R/I$ to answer a question of Grothendieck about vanishing of local cohomology with support in $I$. Note, by definition $R$ is weakly $F$-nilpotent if and only if $\fdp R = \dim R$ and generalized weakly $F$-nilpotent if and only if $\gfdp R=\dim R$. We record several basic facts about $F$-depth found in \cite{Lyu06} and generalized versions which help demonstrate that depth and (generalized) $F$-depth act analogously. The authors are grateful to T. Polstra for discussion leading to a proof of part of the third claim.

\begin{lem}\label{lem:Fdepthbasic}
We have the following bounds on $\fdp R$ and $\gfdp R$.
\begin{enumerate}[label=(\alph*)]
\item $0\le \fdp R\le \dim R$, and $\fdp R>0$ if $\dim R>0$
\item $\fdp R=\fdp\widehat{R}$ and unlike ordinary depth, $\fdp R=\fdp R/\sqrt{0}$
\item $\gfdp R \ge \fdp R$, and $\gfdp R \le \dim R$ if $R$ is equidimensional
\item $\gfdp R=\gfdp R/\sqrt{0} = \gfdp \widehat{R}$
\end{enumerate}
\end{lem}
\begin{proof}
The first two claims are proven in Section 4 of \cite{Lyu06}. The first claim in (c) is obvious, since nilpotent modules are also generalized nilpotent. Now to the second claim in (c). For convenience, we set \[H(\fp) = H^{\operatorname{ht} \fp}_{\fp R_\fp}(R_\fp)/0^F_{H^{\operatorname{ht} \fp}_{\fp R_\fp}(R_\fp)}\] for any $\fp \in \operatorname{Spec}(R)$, and write $H=H(\fm)$. By (a), we know $H(\fp)\neq 0$ for all $\fp \in \Spec(R)$. Denoting by $(\Arg)^\vee$ the Matlis dual over $R$, \cite[Lem. 5.1]{KMPS} gives that $H^\vee$ localizes to the Matlis dual over $R_\fp$ of $H(\fp)$. Thus, $\supp(H)=\supp(H^\vee)=\Spec(R)$ and so $H$ cannot be finite length unless $d=0$, in which case $\gfdp R=0$ for free.

Now for the first claim in (d), consider the natural short exact sequence of finitely-generated $R[F]$-modules below: \begin{center}
\begin{tikzcd}
0 \arrow{r} & \sqrt{0} \arrow{r} & R\arrow{r} & R' \arrow{r} & 0
\end{tikzcd}
\end{center} where $R'=R/\sqrt{0}$. Since $\sqrt{0}$ is a finitely-generated nilpotent $R[F]$-module, $H^j_\fm(\sqrt{0})$ is also nilpotent, and hence generalized nilpotent, for all $j$. Thus, $\gfdp \sqrt{0}=\infty$ and so $\gfdp R=\gfdp R'$ by the later theorem \ref{thm:Fdepthses}.  

For the final claim in (d), write $(S,\fn)=(\widehat{R},\widehat{\fm})$, notably faithfully flat over $R$ with $\sqrt{\fm S}=\fn$. The Frobenius map $F \colon R\rightarrow R$ becomes the Frobenius map $F \colon S\rightarrow S$ upon tensoring with $S$. Furthermore, $H^j_\fm(R)\otimes_R S \simeq H^j_\fm(S) \simeq H^j_\fn(S)$ for any $j$ by the change-of-rings property of local cohomology. Thus, for $e\gg 0$ and any $j\in\NN$ the exact sequence of $R[F]$-modules given below \begin{center}
\begin{tikzcd}
0 \arrow{r} & 0^F_{H^j_\fm(R)} \arrow{r} & H^j_\fm(R) \arrow{r}{F^e} & H^j_\fm(R) 
\end{tikzcd}
\end{center} becomes the following the exact sequence of $S[F]$-modules after applying $\bullet\otimes_R S$. \begin{center}
\begin{tikzcd}
0\arrow{r} & 0^F_{H^j_\fn(S)} \arrow{r} & H^j_\fn(S) \arrow{r}{F^e} & H^j_\fn(S)
\end{tikzcd}
\end{center} Hence, for any $j \in \NN$, we have $J=\Ann_R H^j_\fm(R)/0^F_{H^j_\fm(R)}$ is $\fm$-primary if and only if $JS = \Ann_S H^j_\fn(S)/0^F_{H^j_\fn(S)}$ is $\fn$-primary, so we have $\gfdp R=\gfdp S$.
\end{proof}

One of the most useful and fundamental features of depth, natural inequalities along short exact sequences, is immediately verifiable for (generalized) $F$-depth as well. This theorem will serve as a fundamental tool in the analysis of nilpotent singularity types that follows.

\begin{thm}\label{thm:Fdepthses}
Suppose the following is a short exact sequence of finitely generated $R[F]$-modules. \begin{center}
\begin{tikzcd}
0 \arrow{r} & A \arrow{r} & B \arrow{r} & C \arrow{r} & 0
\end{tikzcd}
\end{center} We then have the following $F$-depth bounds. \begin{enumerate}
\item $\fdp B \ge \min\{\fdp A,\fdp C \},$
\item $\fdp A \ge \min\{\fdp B, \fdp C +1\},$ and
\item $\fdp C \ge \min\{\fdp B, \fdp A -1\}.$
\end{enumerate} If the sequence is split, then $\fdp B =\min\{\fdp A,\fdp C\}$. Finally, the same inequalities and equality hold replacing $\fdp$ with $\gfdp$.
\end{thm}

\begin{proof}
We will prove (2), the others are similar. Suppose $t < \min\{ \fdp B,\fdp C +1\}$ and we will show $H^j_\fm(A)$ is nilpotent for $j \le t$. Since we have the long exact sequence of $R[F]$-modules \begin{center}
\begin{tikzcd}
\cdots \arrow{r} & H^{j-1}_\fm(C) \arrow{r}{\delta} & H^j_\fm(A) \arrow{r}{\alpha} & H^j_\fm(B) \arrow{r} & \cdots
\end{tikzcd}
\end{center} we know that the image of $H^j_\fm(A)$ is nilpotent as $j<\fdp B$ so since $\hsl H^j_\fm(B)=e<\infty$, $\rho^e_{H^j_\fm(A)}(H^j_\fm(A))$ is inside $\ker(\alpha)=\im(\delta)$. But, $\im(\delta)$ is a subquotient of $H^{j-1}_\fm(C)$ and $j-1<\fdp C$ and is thus nilpotent by Lemma~\ref{lem:NilSES}. Consequently, $\rho^e_{H^j_\fm(A)}(H^j_\fm(A))$ is nilpotent which implies $H^j_\fm(A)$ is also nilpotent, as required. 

In the case that the sequence is split, we have $H^j_\fm(B) = H^j_\fm(A)\oplus H^j_\fm(C)$ and consequently, $H^j_\fm(B)$ is nilpotent if and only if both $H^j_\fm(A)$ and $H^j_\fm(C)$ are by Lemma~\ref{lem:NilSES}. This shows the required $F$-depth bound. 

To see the final claim, repeat the above proof with the generalized nilpotent part of Lemma~\ref{lem:NilSES}, noting that the modules in question have finite Hartshorne-Speiser-Lyubeznik numbers as local cohomology modules of finitely generated modules are artinian. 
\end{proof}

\section{Gluing Nilpotent Singularities}\label{sec:glue}

We recall now a way to construct local rings with prescribed properties. A property $\cP$ of local rings is said to \textbf{glue} if, given ideals $\fa_1$ and $\fa_2$ of $R$ with $\fb=\fa_1+\fa_2$, whenever $R/\fa_1$, $R/\fa_2$, and $R/\fb$ satisfy $\cP$, then so does $R/\fa_1\cap\fa_2$. In geometric language, this asks if $Y_1$ and $Y_2$ are subschemes of a $X$ so that each local ring of $Y_1$, $Y_2$, and $Y_1\cap Y_2$ have $\cP$, then each local ring of $X$ has $\cP$. Several $F$-singularity types are known to glue, a partial list is given below. 

\begin{itemize}
\item Schwede showed in \cite[Prop.~4.8]{Sch09} that Cohen-Macaulay $F$-injective singularities glue.
\item Quy and Shimomoto showed in \cite[Thm.~5.7]{QS17} that \textit{stably $FH$-finite} singularities glue.
\item Dao-De Stefani-Ma showed in \cite[Prop.~2.8]{DDSM} that a characteristic-agnostic singularity type called \textit{cohomologically full} singularities glue, which specializes to the \textit{$F$-full} singularities of Ma-Quy, \cite{MQ18} in prime characteristic.
\end{itemize} This makes the gluing question a natural one to ask when attempting to understand the geometry of a new $F$-singularity type and a useful tool for constructing examples. We investigate how nilpotent singularity types glue.

\begin{run} For the remainder of this section, we need to establish several conditions on the ideals $\fa_1,\fa_2$, and $\fb=\fa_1+\fa_2$, which we collect here. Fix $(R,\mfm)$ a local ring of dimension $d\ge 2$ (to avoid trivialities), and let $\fa_1$ and $\fa_2$ be ideals of $R$. We may freely work modulo $\fa_1\cap\fa_2$ to assume $\fa_1\cap\fa_2=0$ in $R$.  We will also specialize to the case that $d=\dim R/\fa_1=\dim R/\fa_2$ as in \cite[Thm.~5.7]{Sch09}, although a finer analysis is possible.\end{run}
\subsection{Gluing nilpotent singularity types} We come to our first main theorem constructing nilpotent like singularities which we phrase as a bound on $F$-depth under gluing.

\begin{thm}
We have the following lower bound on the $F$-depth of $R$, and the same lower bound holds replacing $\fdp$ with $\gfdp$ throughout. \[
\fdp R \ge \min \{\fdp R/\fa_1 ,\fdp R/\fa_2,\fdp R/\fb+1\}
\]
\end{thm}

\begin{proof}
We may analyze the $F$-depth of $R$ in terms of $R/\fa_i$ and $R/\fb$ using the natural short exact sequence below, which we augment with the canonical Frobenius map in each place. \begin{center}
\begin{tikzcd}
0 \arrow{r} & R\arrow{d}{F} \arrow{r}{\alpha} & R/\fa_1\oplus R/\fa_2 \arrow{d}{F\oplus F}\arrow{r}{\beta} & R/\fb\arrow{d}{F} \arrow{r} & 0 \\
0 \arrow{r} & R \arrow{r}{\alpha} & R/\fa_1\oplus R/\fa_2 \arrow{r}{\beta} & R/\fb \arrow{r} & 0 
\end{tikzcd}
\end{center} The maps $\alpha(r)=(r+\fa_1,r+\fa_2)$ and $\beta(r+\fa_1,s+\fa_2)=r-s+\fb$ commute with the Frobenius actions given in each place, which means we can view the top row in the diagram above as a short exact sequence of $R[F]$-modules as in Lemma~\ref{thm:Fdepthses}. Notably, the (generalized) $F$-depth of $R/\fa_i$ as an $R$-module is the same as its (generalized) $F$-depth as an $R/\fa_i$-module, since $\mfm R/\fa_i$ is the maximal ideal of $R/\fa_i$. By applying the lemma, we get the required lower bounds. 
\end{proof}

\noindent As an immediate corollary of this lemma we can glue (generalized) weak $F$-nilpotence. 

\begin{cor}\label{thm:GlueWeakFNil}
Suppose $\dim R/\fb\ge d-1$. If $R/\fa_1,R/\fa_2$, and $R/\fb$ are weakly $F$-nilpotent, so is $R$. Further, if $R$ is equidimensional and $R/\fa_1$, $R/\fa_2$, and $R/\fb$ are generalized weakly $F$-nilpotent, so is $R$.
\end{cor}

\begin{proof}
The conditions provided and the proceeding lemma imply the (generalized) $F$-depth of $R$ is at least $d$, but the (generalized) $F$-depth of $R$ is at most $d$ as well, completing the claim.
\end{proof}

\begin{rmk}\label{rmk:cannot glue wFn with low-dimensional intersection}
We note that, up to the assumptions on $R/\fa_1$ and $R/\fa_2$, the conditions in \ref{thm:GlueWeakFNil} on $R/\fb$ are sharp. In particular, if we write $t=\fdp R/\fb$ and we have \[t+1<\min\{\fdp R/\fa_1 ,\fdp R/\fa_2\}=s,\] then $\fdp R=t+1$. To see this, let $A_j = H^j_\mfm(R/\fa_1)\oplus H^j_\mfm(R/\fa_2)$. By Lemma~\ref{thm:Fdepthses}, $A_j$ is nilpotent for all $j\le t+1$. Thus, we get an exact sequence of $R[F]$-modules \begin{center}
\begin{tikzcd}
A_t \arrow{r} & H^t_\mfm(R/\fb) \arrow{r}{\delta} & H^{t+1}_\mfm(R) \arrow{r} & A_{t+1}
\end{tikzcd}
\end{center} where Lemma~\ref{lem:NilSES} assures that $\im(\delta)\subset H^{t+1}_\mfm(R)$ is not nilpotent. Thus, $\fdp R\le t+1$, which forces $\fdp R=t+1$ since $\fdp R\ge \min\{s,t+1\}=t+1$ by Lemma~\ref{thm:Fdepthses}. A similar obstruction occurs in the generalized case.
\end{rmk}

We are grateful to A. Simpson for discussions about gluing and irreducibility which contributed greatly to the remark below.

\begin{rmk}
Given Corollary~\ref{thm:GlueWeakFNil}, it seems natural to ask in the same setting whether $F$-nilpotent singularities glue as well. In upcoming work of the first author with H. Dao and V. Pandey, it will be shown that if $(R,\mfm)$ is an $F$-finite, $F$-nilpotent local ring, then $R$ has a unique minimal prime, so that no non-trivial gluing can occur since $\Spec(R)$ is irreducible. For the same reason, any $F$-singularity type which forces the ring to be a domain will not glue, e.g. $F$-rational or weakly $F$-regular. This remark replaces a vacuous gluing statement for $F$-nilpotence from a prior version of this work.
\end{rmk}

\subsection{Examples}

We now demonstrate Corollary~\ref{thm:GlueWeakFNil} with some explicit examples.

\begin{xmp}\label{xmp:gluewFn}
Let $S = k[x,y,z,w]$, $\fa_1 = (xw-yz,yw^2-z^3,xz^2-y^2w,x^2z-y^3)$ and $\fa_2 = (x,z)$. We now check the conditions of Corollary~\ref{thm:GlueWeakFNil}, suppressing the process of localizing at the origin. 

\begin{itemize}
\item The ring $S/\fa_1$ is isomorphic to $k[s^4,s^3t,st^3,t^4]$, a two-dimensional domain. The first local cohomology module of this ring is well-known to be nilpotent under Frobenius, see \cite[Example~2.5.1]{MQ18}, so $S/\fa_1$ is weakly $F$-nilpotent.
\item The ring $S/\fa_2$ is a two-dimensional Cohen-Macaulay ring, so is weakly $F$-nilpotent.
\item Finally, it is easily seen that $\fb = (x,z,yw^2,y^2w,y^3)$ and $S/\fb \cong k[y,w]/(yw^2,y^2w,y^3)$ is a one-dimensional ring, so is weakly $F$-nilpotent. 
\end{itemize}
Thus, the two-dimensional ring: \[R = S/(\fa_1 \cap \fa_2) = S/(yz-xw,z^4-xw^3,xz^3-xyw^2,x^2z^2-xy^2w,xy^3-x^3z)\] is also weakly $F$-nilpotent after localizing at $(x,y,z,w)R$. 
\end{xmp}

By adjusting the ideal $\fa_2$ in the example above, large families of examples can be easily generated. We now turn to the generalized nilpotent case. For completeness, we also provide an explicit example of (non-trivial) gluing of generalized weak $F$-nilpotence.

\begin{xmp} 
Let $S = k[a,b,c,d,e,f]$, and we will again suppress the process of localizing at the origin.
\begin{itemize}
\item Let $\fa_1$ be the ideal defining the natural map from $S$ to the Segre product $k[x,y,z]/(x^4+y^4-z^4) \# k[u,v]$; Example~\ref{xmp:Kunneth} will show that the three dimensional ring $S/\fa_1$ is generalized weakly $F$-nilpotent (but not weakly $F$-nilpotent). 

\item Let $\fa_2 = (cd-ae,c^2-bd,b^2-ac)$, and it is immediate to verify that $S/\fa_2$ is a three-dimensional Cohen-Macaulay ring. In particular, $S/\fa_2$ is also generalized weakly $F$-nilpotent. 

\item For a particular choice of generators for $\fa_1$, we get \[\fb =  (ce-bf, ae-af,cd-af,bd-af,c^2-af,b^2-ac,d^4+e^4-f^4, ad^3+be^3-cf^3, a^2d^2,a^3d,a^4)\] and $S/\fb$ is a two dimensional ring. It is easy to check that $S/\sqrt{b}=S/(c,b,a,d^4+e^4-f^4)$ is Cohen-Macaulay, so by Lemma~\ref{lem:Fdepthbasic}, $S/\fb$ is weakly $F$-nilpotent. 
\end{itemize}
Since $R=S/(\fa_1\cap \fa_2)$ a three-dimensional ring, we see $R$ is generalized weakly $F$-nilpotent after localizing at $(a,b,c,d,e,f)R$ by Corollary~\ref{thm:GlueWeakFNil}. 
\end{xmp}

\subsection{Frobenius test exponents for gluing}

We can directly apply Remark~\ref{rmk:finiteness theorems} to bound the Frobenius test exponent of the glued ring produced in Corollary~\ref{thm:GlueWeakFNil}. We use the running assumptions in Section~\ref{sec:glue}, and for convenience, we will also write $h_j=\max\{\hsl H^j_\mfm(R/\fa_1),\hsl H^j_\mfm(R/\fa_2) \}$.

\begin{thm}\label{thm:HSLGluing}
For $j\le \fdp R/\fb$, we have: \[\hsl H^j_\mfm(R) \le \hsl H^{j-1}_\mfm(R/\fb) + h_j.\] In particular, if $\dim R/\fb =d$ and $R/\fa_i$ and $R/\fb$ are weakly $F$-nilpotent, this bound holds for all $j$. If $\dim R/\fb =d-1$, then the bound holds for $0 \le j < d$. This implies the following Frobenius test exponent bounds when $R/\fb$ and $R/\fa_i$ are weakly $F$-nilpotent: \begin{itemize}
\item $\displaystyle \fte R \le \sum_{i=1}^{d} \binom{d}{j} \hsl H^{j-1}_\mfm(R/\fb)  + \sum_{i=0}^d \binom{d}{j}h_j$ if $\dim R/\fb =d$.
\item $\displaystyle \fte R \le \hsl H^d_\mfm(R)+ \sum_{i=1}^{d-1} \binom{d}{j} \hsl H^{j-1}_\mfm(R/\fb) + \sum_{i=0}^{d-1} \binom{d}{j}h_j$ if $\dim R/\fb =d-1$.
\end{itemize}
\end{thm}

\begin{proof}
We may apply Lemmas~\ref{lem:lift of nilp is nilp hsl bounds}  and \ref{lem:lift of nilp is nilp hsl bounds} to see the first set of bounds. The Frobenius test exponent bounds then follow from Quy's upper bound and Corollary~\ref{thm:GlueWeakFNil}. 
\end{proof}

We note that in each case of the theorem above, the summands in the upper bound are involved in computing Quy's upper bound on $R/\fa_i$ and $R/\fb$, but if $\dim R/\fb =d-1$, we need to include $\hsl H^d_\mfm(R)$ due to the possible failure of the condition in Lemma~\ref{lem:lift of nilp is nilp hsl bounds}. A similar theorem giving an upper bound on $\fte R$ if $R/\fb$ and $R/\fa_i$ are generalized weakly $F$-nilpotent is possible, but its increased technicality which makes it of less interest for the following reason. Further, without the additional (and computationally ineffective) assumption that $\im(\delta:H^{j-1}_\mfm(R/\fb)\rightarrow H^j_\mfm(R))$ is nilpotent for $j>\fdp R/\fb$, the Hartshorne-Speiser-Lyubeznik bounds in Theorem~\ref{thm:HSLGluing} will not apply.

\section{Modules with Frobenius actions: the graded case}\label{sec:graded}

We now turn to the setting when the Frobenius map on a graded ring induces a nilpotent Frobenius action on its graded local cohomology modules. To be clear, by a \textbf{graded ring} $R$, we mean a noetherian $\NN$-graded ring with degree zero part $[R]_0$ a field, and a \textbf{standard graded} ring is an graded ring where $R$ is generated as an $[R]_0$-algebra by $[R]_1$. For such a ring, we write $\fm=\fm_R=\oplus_{i>0} [R]_i$ for its unique homogeneous maximal ideal. The reader is referred to \cite{GW78a} for the construction of graded local cohomology modules and other details about graded rings. We also recommend \cite[Section~2]{Sin02} for a detailed review of the nearly all of the necessary graded constructions we use.

\begin{run}For the remainder of Section~\ref{sec:graded}, let $R$ be a graded ring with homogeneous maximal ideal $\fm$, where $[R]_0=k$ is a field of prime characteristic $p>0$. We write $d=\dim R$ and we will assume $d>0$. For a homogeneous ideal $J\subset R$, we write $H^j_J(\bullet)$ for the graded local cohomology supported in $J$. Finally, we let $M$ be a $\ZZ$-graded $R$-module.
\end{run}

Notice in a graded ring, tight and Frobenius closure of homogeneous ideals are homogeneous. Further, note that known results about finiteness of Frobenius test exponents will still apply to the graded setting as long as we restrict to homogeneous parameter ideals, which can be generated by homogeneous filter regular elements by homogeneous prime avoidance.

\begin{dff}\label{dff:homogeneous fte}
If $R$ is a graded ring, then the {\bf homogeneous Frobenius test exponent of $R$}, written $\fte^* R$, is defined as \[
\fte^* R=\sup\{\fte \fp \mid \fq \subset R \text{ a homogeneous parameter ideal}\}\in\NN\cup\{\infty\}.\]
\end{dff}

\begin{rmk}\label{rmk:homogeneous fte works like fte}
A key technique in studying the finiteness of Frobenius test exponents is the use of filter regular sequences and the Nagel-Schenzel isomorphism, see \cite[Rmk. 2.5]{Quy19}. A graded version of the Nagel-Schenzel isomorphism for graded local cohomology modules is also readily available, for example, by following the usual graded adjustments to the proof provided by Huong in \cite{Huo17}. 
\end{rmk}

We also have need to study and manipulate the set of degrees for which a graded module is nonzero. 

\begin{dff}
Define the \textbf{degree support of $M$} to be the set $\degsup M=\{n\in \ZZ \mid [M]_n\neq 0\}$. If $M'$ is another $\ZZ$-graded $R$-module, then we say \textbf{$M$ and $M'$ have complementary degree supports} if $\degsup M\cap \degsup M'=\varnothing$. If we are considering multiple gradings $g_1,g_2$ on the same module we may write $\degsup^{g_1} M$ and $\degsup^{g_2} M$ to distinguish them.
\end{dff}

Notice that the degree support of $M$ only depends on the underlying direct sum decomposition as abelian groups given by the grading $M=\bigoplus [M]_n$.

\subsection{Graded Frobenius actions}

We aim to study the Frobenius action on the graded local cohomology modules of $R$, which are naturally $\ZZ$-graded. To accomplish this, we need several basic results about the correct notion of Frobenius actions on graded $R$-modules. First, note the Frobenius map $F \colon R\rightarrow R$ enjoys the property that $F([R]_n)\subset [R]_{np}$, and so we require this of a Frobenius action on a graded $R$-module. 

\begin{dff}
We say that $(M,\rho)$ is a \textbf{graded $R[F]$-module} or has a \textbf{graded Frobenius action} if $\rho \colon M\rightarrow M$ is a $p$-linear map with the property that $\rho([M]_n)\subset [M]_{np}$ for all $n \in \ZZ$. 
\end{dff}

The definitions in Definition~\ref{dff:R[F]-module basics} apply equally well to graded $R[F]$-modules, where we require $\rho$-stable submodules to be graded submodules. 

\begin{xmp}
For any graded $R[F]$-module $(M,\rho)$, note that $[M]_0$ is a graded $\rho$-stable submodule of $M$ as $\rho([M]_0)\subset [M]_{p\cdot 0}=[M]_0$. Furthermore, if $J$ is a homogeneous ideal of $R$, the graded local cohomology modules $H^j_J(M)$ are naturally graded $R[F]$-modules, where the Frobenius action $\rho \colon H^j_J(M)\rightarrow H^j_J(M)$ is the map induced by $\rho:M\rightarrow M$. All mentions of graded $R[F]$-structures on the graded local cohomology of a graded $R[F]$-module will be this natural one.
\end{xmp}

\begin{dff}\label{dff:gradedNilpotent}
Let $(M,\rho)$ be a graded $R[F]$-module. For $n \in \ZZ$, call $M$ {\bf nilpotent in degree $n$} provided $[M]_n = \cup_e \ker(\rho^e:[M]_n \rightarrow [M]_{np^e}).$
\end{dff}

\noindent It is trivial to see that $M$ is nilpotent if and only if $M$ is nilpotent in degree $n$ for all $n \in \ZZ$. In particular, $0^\rho_M$ is a graded $\rho$-stable submodule of $M$ since \[
0^\rho_M = \bigoplus_{n\in\ZZ} \left(\bigcup_{e \ge 0} \ker(\rho^e \colon [M]_n \rightarrow [M]_{np^e})\right).
\] Due to this, we can study nilpotence in degree $n$ by studying the degree support of the graded $R[F]$-module quotient $M/0^\rho_M$, which gives rise to the following definition.

\begin{dff}\label{dff: nilsupp}
Let $M$ be a graded $R[F]$-module. Define the \textbf{nil-support of $M$} to be the set of $$\nilsup M = \degsup M/0^\rho_M = \{ n \in \ZZ \colon M \text{ is not nilpotent in degree } n\}.$$ If $M'$ is another graded $R[F]$-module, then we say \textbf{$M$ and $M'$ have complementary nil-supports} if $\nilsup M \cap \nilsup M' = \varnothing$. If we need to be clear about which grading $g$ or action $\rho$ we are working with we will write $\nilsup^g_\rho M = \degsup^g M/0^\rho_M$.
\end{dff}

We have $\degsup N \cup \degsup M/N \subset \degsup M$ for all graded $R$-submodules $N\subset M$. In particular, $\nilsup M \subset \degsup M$.

\begin{lem}\label{lem:graded nilp and gen nilp} \label{lem:trichotomy}
Let $M$ be a graded $R[F]$-module such that $\dim_k [M]_n$ is finite for all $n\in\ZZ$. The following trichotomy holds for $\nilsup M$. \begin{enumerate}[label=(\arabic*)]
\item $\nilsup M = \varnothing$ i.e., $M$ is nilpotent. 
\item $\nilsup M = \{0\}$, equivalently that $M$ is generalized nilpotent but not nilpotent.
\item $\nilsup M$ is an infinite subset of $\ZZ$.
\end{enumerate}
\end{lem}

\begin{proof}
Since the zero module is the only module with empty degree support, (1) is immediate. If $j \in \nilsup M$, then $jp^e\in \nilsup M$ for all $e\in \NN$ since the induced Frobenius action on $M/0^\rho_M$ is injective. Thus, if $[M]_0$ is finite dimensional over $k$ then $M$ is generalized nilpotent if and only if it $M/0^\rho_M$ is supported only possibly in degree $0$. 
\end{proof}

\noindent If $\degsup M$ is finite, we can use the above lemma and the grading on $M$ to bound $\hsl M$.

\begin{cor}\label{cor:finite support hsl bound}
Let $M$ be a graded $R[F]$-module with $\degsup M$ a finite subset of $\ZZ$. If $e_0\in \NN$ is minimal so that $\degsup M \subset (-p^{e_0},p^{e_0})$ then $\hsl M\le \min\{\hsl [M]_0,e_0\}$. 
\end{cor}

\begin{rmk}\label{rmk:infinite nilsup is fucking annoying}
It seems rather hard in general to determine the nil-support of $M$ when it is infinite. An effective algorithm to compute $\nilsup M$ would be of great use in understanding nilpotent singularity types and computing Frobenius test exponents.
\end{rmk}

\subsection{Graded nilpotent singularities}

At this point, it should be clear that the (generalized) $F$-depth of an $R[F]$-module over a local ring can be extended to the graded setting if we  measure the nilpotence the graded local cohomology supported in the homogeneous maximal ideal. The definitions of (generalized) $F$-depth, (generalized) weakly $F$-nilpotent, and $F$-nilpotent carry over verbatim for a graded $R[F]$-module using graded local cohomology. 


\begin{rmk}\label{rmk:graded case is similar}
The results of Subsections~\ref{sec:R[F]-modules} and \ref{sec:(g)F-depth} extend to the graded versions since the commutative diagrams  are largely identical.  Furthermore, as noted in Remark~\ref{rmk:homogeneous fte works like fte} in hand, the finiteness theorems given in \ref{rmk:finiteness theorems} will apply to $\fte^* R$ when $R$  is a graded local ring satisfying the graded analogues of the stated hypotheses. 
\end{rmk}

To help track nilpotence in degree 0, we use the following definition which is a nilpotent version by the classical $a$-invariant of a module.

\begin{dff}
Let $M$ be a graded $R[F]$-module $M$, we define the \textbf{base-nilpotent index of $M$} as \[b(M) := \sup\{n \in \ZZ \mid M \text{ is not nilpotent in degree } n\} \in \ZZ \cup \{-\infty\}.\] 
\end{dff}

\noindent Note that \textbf{$a$-invariant of $M$} is $
a(M)=\sup \left\lbrace \degsup M\right\rbrace \in \ZZ \cup \{\pm \infty\}.$
  As for the $a$-invariant, $b(M)$ is can be phrased in terms of the nil-support \[
b(M)=\sup\left\lbrace \nilsup M \right\rbrace \in \ZZ\cup\{\pm \infty\}
\] and trivially $b(M) = -\infty$ if and only if $M$ is nilpotent. 

For $j \geq 0$, we define \textbf{$j$-th base-nilpotent index of $R$} to be $b_j(R) = b(H^j_{\fm}(R))$. This is analogous to the $i$-th $a$-invariant of $R$, $a_i(R) := a(H_\fm^i(R))$. Finally, we have need to track the first cohomological index where $b_j(R)$ is zero, so we set $$b(R) = \inf\{ j\in \NN \mid b_j(R)=0 \} \in \NN\cup\{\infty\}.$$

\begin{rmk}\label{rmk:basenilpindex}
We record several observations about this sequence of invariants.
\begin{enumerate}
\item The graded $R[F]$-module $H^j_{\fm}(R)$ is nilpotent if and only if $b_j(R)=-\infty$.
\item For all $j$,  $b_j(R) \le 0$ as $H^j_\fm(R)$ is artinian, whence nilpotent in positive degree. 
\item If $a(H^j_\fm(R))<0$ for all $0 \le j \le d$, then $b_j(R)<0$ for all $0\le j\le d$, which implies $b(R)=\infty$.
\item Generally $b(R) \ge \fdp R$, since $b_j(R) = -\infty$ for $j < \fdp R$. If $R$ is weakly $F$-nilpotent, then $b(R) \ge \dim R$.
\item Further, if we have $f=\fdp R<\gfdp R$, then $H^f_{\fm}(R)$ is not nilpotent but is generalized nilpotent, whence $b_{f}(R)=0$ and $b(R) \leq f$. 
\end{enumerate}
\end{rmk}

We have an immediate corollary to Lemma~\ref{lem:graded nilp and gen nilp}. Note the lemma applies to the graded local cohomology modules of $R$, as graded local duality ensures that $[H^j_\fm(R)]_n$ is a finite-dimensional $k$-vector space for all $n \in \ZZ$.

\begin{cor}\label{cor:gwfn and b(R)=infty implies wFn}
For $R$ a generalized weakly $F$-nilpotent graded ring, $R$ is weakly $F$-nilpotent if and only if $b_j(R)<0$ for $0\le j<d$.  
\end{cor}

The nilpotence of the degree 0 part of an artinian module with a Frobenius action was studied in \cite{LSW16} over $\mathbf{F}_p[x_1,\ldots,x_n]$, where connections were made with the composition series for the canonical $\mathcal{D}$-module and $\mathcal{F}$-module structures, see \cite[Thm.~2.9]{LSW16}. We avoid the language of $\mathcal{F}$-modules here in favor of a more elementary approach.

\begin{rmk}\label{rmk:degree0nilpotent} Note that $b(R)=\infty$ is a necessary condition for $R$ to be $F$-nilpotent. Certainly when $R$ is $F$-nilpotent, it is weakly $F$-nilpotent, and thus $b_j(R)=-\infty$ for $j<d$. Furthermore, $[H^d_\fm(R)]_0$ is in $0^*_{H^d_\fm(R)}$ since $cF^e([H^d_\fm(R)]_0)\subset c[H^j_\fm(R)]_0=0$ for $\deg(c)\gg 0$. Thus, if $0^*_{H^d_\fm(R)}=0^F_{H^d_\fm(R)}$, we get $[H^d_\fm(R)]_0$ is nilpotent and so $b_{d}(R) \neq 0$, whence $R$ being $F$-rational forces $b(R) = \infty$.
\end{rmk}

\noindent In some cases, $b(R)=\infty$ can also be sufficient to ensure that $R$ is $F$-nilpotent.

\begin{lem}\label{lem:punctured spec}
If $R$ is $F$-rational on the punctured spectrum, then $R$ is $F$-nilpotent if and only if $b(R) = \infty$.
\end{lem}
\begin{proof} 
If $R$ is $F$-nilpotent, by we saw $b(R)=\infty$ in Remark~\ref{rmk:degree0nilpotent} so it remains to prove the converse. By \cite[Ex. 2.7(1)]{ST15}, $F$-rational on the punctured spectrum forces $R$ to be generalized Cohen-Macaulay, whence generalized weakly $F$-nilpotent, and $0_{H_\fm^d(R)}^*$ has finite length.  If $b(R)=\infty$, then clearly $b_j(R) < 0$ for all $j \leq \dim R$ and so Corollary~\ref{cor:gwfn and b(R)=infty implies wFn} shows that $R$ is weakly $F$-nilpotent. Further, by Lemma~\ref{lem:graded nilp and gen nilp}, $\nilsup 0^*_{H^d_\fm(R)} = \varnothing$ which forces $0^*_{H^d_\fm(R)} = 0^F_{H^d_\fm(R)}$.
\end{proof}

\begin{rmk}
Along a similar vein, if $R$ is equidimensional and $F$-nilpotent on the punctured spectrum, then \cite[Prop.~3.1 and Lem.~5.1]{KMPS} imply that $R$ is generalized weakly $F$-nilpotent and $0^*_{H^d_\fm(R)}/0^F_{H^d_\fm(R)}$ is finite length. The same proof of Lemma~\ref{lem:punctured spec} shows $R$ is $F$-nilpotent if and only if $b(R)=\infty$ in this case also.
\end{rmk}

To provide some concrete examples, we will use Fermat hypersurfaces. Our first example emphasizes how the $F$-nilpotent property is characteristic-dependent.

\begin{xmp}\label{xmp:Blickle}
The explicit calculation in \cite[Ex. 5.28]{Bli01} shows that the nilpotence of the Cohen-Macaulay ring $R = k[x,y,z]/(x^4 + y^4 - z^4)$ is dependent on the characteristic of $k$ modulo $4$. When $k$ is of characteristic $p \equiv 3 \bmod 4$, $[H^2_\fm(R)]_0$ is nilpotent so $b_2(R)<0$ and when $p \equiv 1 \bmod 4$, $[H^2_\fm(R)]_0$ is not nilpotent so $b_2(R)=0$.  Consequently, if $p \equiv 3 \bmod 4$, then $R$ is $F$-nilpotent, and if $p \equiv 1 \bmod 4$, then the Frobenius action is injective on $[H_\fm^2(R)]_0\subset 0^*_{H^2_\fm(R)}$, so $R$ is not $F$-nilpotent. See Lemma~\ref{lem:FermatHyp} for a family of similar examples. 
\end{xmp}

To generalize Example~\ref{xmp:Blickle}, we give a sufficient condition for degree $d$ Fermat hypersurface \[R = k[x_0,\ldots,x_n]/(x_0^d + \cdots + x_{n-1}^d - x_n^d)\] to be $F$-nilpotent. This ring is Cohen-Macaulay, so it is already weakly $F$-nilpotent and trivially $F$-nilpotent on the punctured spectrum. By Remark~\ref{rmk:degree0nilpotent}, we can determine a condition for when $H_\fm^n(R)$ is nilpotent in degree $0$ which will ensure the ring is $F$-nilpotent.

\begin{lem}\label{lem:FermatHyp} For $R = k[x_0,\ldots,x_n]/(x_0^d + \cdots + x_{n-1}^d - x_n^d)$, if  $p > d$, $n \geq 2$, and $p \equiv -1 \bmod d$, then $H_\fm^n(R)$ is nilpotent in degree $0$. In fact, the Frobenius is zero on $[H_\fm^n(R)]_0$.
\end{lem}
\begin{proof}
Much as in \cite[Lem. 5.26]{Bli01}, the calculation comes down to an explicit \v{C}ech computation. The submodule $[H_\fm^n(R)]_0$ has basis constructed from partitions $I$ with $\ui = (i_0,\ldots,i_n) \in I$ if and only if $i_j \geq 1$ and $\sum_{j=0}^n i_j = d$. In particular, the set of classes $\eta_{\ui} := \left[ \frac{ x_n^{d-i_n} }{  x_0^{i_0} \cdots x_{n-1}^{i_{n-1}}}\right]$ forms the desired basis. Write $(d - i_n)p - i_n = dr$ for integer $r$. Directly we see,

\begin{eqnarray*}
F(\eta_{\ui}) & = & \left[ \frac{ x_n^{(d-i_n)p} }{  x_0^{pi_0} \cdots x_{n-1}^{pi_{n-1}}}\right] \\
 & = & z^{i_n} \left[ \frac{ x_n^{(d-i_n)p - i_n} }{  x_0^{pi_0} \cdots x_{n-1}^{pi_{n-1}}}\right] \\
 & = & z^{i_n} \left[ \frac{ x_n^{dr} }{  x_0^{pi_0} \cdots x_{n-1}^{pi_{n-1}}}\right] \\
 & = & z^{i_n} \left[ \frac{ (x_0^d + \cdots + x_{n-1}^d)^r }{  x_0^{pi_0} \cdots x_{n-1}^{pi_{n-1}}}\right] \\
 & = & z^{i_n} \left[ \sum_{j_0 + \cdots + j_{n-1} = r} c_{j_0,\ldots,j_{n-1}} \frac{ x_0^{dj_0} \cdots x_{n-1}^{dj_{n-1}} }{  x_0^{pi_0} \cdots x_{n-1}^{pi_{n-1}}}\right] \\
\end{eqnarray*} where $c_{j_0,\ldots,j_{n-1}}$ is a multinomial coefficient. For any summand in this expression to be non-zero, it is necessary that $dj_t < p i_t$ for all $0 \leq t \leq n-1$. As $p \equiv -1 \bmod d$, there is no solution for $j_t$ in an equation $dj_t = p i_t - 1$, in fact, that $d j_t < p i_t$ forces $d j_t \leq p i_t - (d-1)$. However, this would mean that: $$\sum_{t = 0}^{n-1} dj_t \leq p \left(\sum_{i = 0}^{n-1} i_t\right) - n(d-1) = (d - i_n)p - n(d-1).$$ This creates a contradiction, as $\sum_{t = 0}^{n-1} dj_t = dr = (d - i_n)p - i_n$, thus $(d - i_n)p - i_n \leq (d - i_n)p - n(d-1)$ then forces $n (d-1) \leq i_n < d$ which is clearly false. Therefore, for some $t$, $dj_t \geq p i_t$ and so each term in the expression for $F(\eta_{\ui})$ is zero. 
\end{proof}

It seems likely that the Frobenius is nilpotent in degree zero on $H_\fm^n(R)$ for any $p \not\equiv 1 \bmod d$. Indeed, this is asserted and easily shown for $d = 5$ in \cite[Ex. 5.28]{Bli01}. Establishing this fact  requires a finer \v{C}ech computation than we give here, but we will not remark further on these intermediate cases. We end this subsection by describing a sufficient ideal-theoretic condition ensure the top local cohomology module is nilpotent in degree $0$. 

\begin{thm} Let $R$ be standard graded and let $\ux=x_1,\ldots,x_d$ be a homogeneous system of parameters with $\deg(x_i) =1$ for all $i$. Write  $x=x_1\cdots x_d$. If for each $t$ there is $s$ so that $[R]_{dt} \subset (\ux^{t+s})^F \colon x^s$, then $H_{\fm_R}^d(R)$ is nilpotent in degree $0$.
\end{thm}

\begin{proof}

Represent a class $\eta$ in $H_{\fm_R}^d(R) \cong \varinjlim_t R/(\ux^t)$ by a representative $[z + (\ux^t)]$. Recall, in the grading for $H_{\fm_R}^d(R)$, $\deg [z + (\ux^t)] = \deg z -dt$. For $e \ge 0$ set $q = p^e$. The natural Frobenius action has $\rho^e(\eta) = [z^q + (\ux^{tq})]$. This is zero if and only if for some $s$, $(x^s z)^q \in (\ux^{(t+s)q})$, i.e., $z \in (\ux^{t+s})^F \colon x^s$. Thus we have $$0_{H_{\fm_R}^d(R)}^F = \{ \eta = [z + (\ux^t)] \colon  \textrm{there is } s \textrm{ so that } z \in ((\ux^{t+s})^F \colon x^s)\}.$$ 

Any such $[z + (\ux^t)]$ has degree $0$ if and only if $\deg z = dt$. By assumption, there is $s$ with $z \in [R]_{dt} \subset ((\ux^{t+s})^F \colon x^s)$ whence $[z+(\ux^t)]\in 0_{H_{\fm_R}^d(R)}^F$.
\end{proof}

\section{Segre Products} \label{sec:segre}

In this section, we study nilpotent singularity types for Segre products of two graded rings defined over the same field. To do so, we study the the natural diagonal Frobenius action on the Segre product of two modules with  Frobenius actions. We especially try to recapture classical results known for the Cohen-Macaulay property of the Segre product, proved in \cite{GW78a}.

\begin{run}
Let $(R,\fm)$ and $(S,\fn)$ standard graded rings such that $[R]_0 = [S]_0 = k$, the same field. Denote by $\fm_R$ and $\fm_S$ the respective homogeneous maximal ideals. We write $T = R \# S$ for the Segre product of $R$ and $S$, defined by $R \# S = \bigoplus_{n \geq 0} [R]_n \otimes_k [S]_n$, where $[T]_0=[R]_0\otimes_k [S]_0 = k\otimes_k k \simeq k$ and denote by $\fm_T$ for its homogeneous maximal ideal. If $\dim R = d_R$ and $\dim S = d_S$, then $d_T := \dim T = d_R + d_S - 1$.  We assume $\depth(R)\ge 2$ and $\depth(S)\ge 2$ to use a simplified form of the K\"unneth formula, described later. Finally, we let $M$ be a $\ZZ$-graded $R$-module and $N$ be $\ZZ$-graded $S$-module. \end{run}
\subsection{Nilpotent singularity types for Segre products}
The Segre product of $M$ and $N$ is defined as for rings, i.e. $M \# N$ is the $\ZZ$-graded $T$-module with $[M \# N]_n = [M]_n \otimes_k [N]_n$. First, we note that $M\# N = 0$ if and only if $M$ and $N$ have complementary degree supports.

\begin{lem}
We have $\degsup M\# N =\degsup M \cap \degsup N$.
\end{lem}
\begin{proof}
The proof hinges on a simpler fact from linear algebra. For finite-dimensional $\ZZ$-graded $k$-vector spaces $V$ and $W$, then $\dim V\otimes_k W = \dim V\cdot \dim W$. Clearly $\degsup M\# N\subset \degsup M\cap\degsup N$ since $[M]_t \otimes_k 0 \simeq 0$. Further, if $t\in\degsup M\cap\degsup N$, then there is a nonzero $m\in [M]_t$ and $n\in [N]_t$, and the $k$-vector space dimension of $\operatorname{Span}_k(m)\otimes_k \operatorname{Span}_k(n)$ is equal to $1\cdot 1=1$, i.e. $0\neq m\otimes n = m\# n \in [M]_t\otimes_k [N]_t$, so $t \in \degsup M\# N$.
\end{proof}

When $(M,\rho_M)$ is a graded $R[F]$-module and $(N,\rho_N)$ is a graded $S[F]$-modules, we consider $M \# N$ as a graded $T[F]$-module under its diagonal Frobenius action denoted $\rho=\rho_M \# \rho_N$. We need to understand when a Segre product of graded modules with Frobenius actions is nilpotent.

\begin{lem}\label{lem:nilSegre}
We have $\nilsup M \# N = \nilsup M \,\cap\, \nilsup N$. In particular, if $[M]_n$ and $[N]_n$ are finite-dimensional over $k$ for all $n\in\ZZ$, then $M\# N$ is generalized nilpotent if and only if $\nilsup M\# N \subset \{0\}$.
\end{lem}

\begin{proof}
Again, we have $\degsup M\# N/0^\rho_{M\# N} \subset \degsup M/0^{\rho_M}_M \cap \degsup 0^{\rho_N}_N$ for free. But if $t \in \degsup M/0^{\rho_M}_M \cap \degsup N/0^{\rho_N}_N$, then there an $n \in [N]_t$ and $m\in [M]_t$ so that $\rho^e_N(n)\neq 0$ and $\rho^e_M(m)\neq 0$ for all $e \in \NN$. But then $\rho^e(m\# n)\neq 0$ for all such $e$, so $t \in \degsup M\# N$.

Now if $\dim_k[M]_n<\infty$ and $\dim_k[N]_n<\infty$ for each $n\in \ZZ$, then $\dim_k [M\# N]_n = (\dim_k [M]_n)\cdot (\dim_k [N]_n)<\infty$ for each $n\in \ZZ$. Thus, $M\# N$ is generalized nilpotent if and only if it is nilpotent in nonzero degree, which is equivalent to the given intersection condition.
\end{proof}

Our next aim is to consider how $F$-depth and generalized $F$-depth behave under Segre products. Under mild vanishing, $H^j_\fm(M) = 0 = H^j_\fn(N)$ for $j=0,1$, which is guaranteed by our running assumptions, we may apply the K\"unneth formula, \cite[Thm.~4.1.5]{GW78a} \begin{equation*}\label{eq:Kunneth} H_{\fm_T}^j(M \# N) \cong  H_{\fm_R}^j(M) \# N \oplus M \# H_{\fm_S}^j(N) \oplus \left( \bigoplus_{r+s=j+1} H_{\fm_R}^r(M) \# H_{\fm_S}^s(N) \right).
\end{equation*} By \cite[Thm.~4.4.4 (i)]{GW78a}, the ``expected" depth formula for Segre products, i.e., the validity of the equation  $\depth(T) = \depth(R)+\depth(S)-1$, is obstructed by the $a$-invariants of $R$ and $S$. We will see that the $b_j$-invariants of $R$ and $S$ provide a similar obstruction for the expected $F$-depth formula. 

\begin{thm}\label{thm:fdepthSegre}
Write $f = \fdp R + \fdp S -1$. We have $\fdp T \ge \min\{b(R),b(S),f\}$. Furthermore, if $\fdp R =b(R)$ and $\fdp S =b(S)$, we have $\fdp T =f$. 
\end{thm}

\begin{proof}
For $j < b(R)$, each $H^j_{\fm_R}(R)$ is nilpotent in degree $0$. As $S$ is only supported in non-negative degree and  $H^j_{\fm_R}(R)$ is nilpotent in non-negative degree, we have by Lemma~\ref{lem:nilSegre} that $H^j_{\fm_R}(R)\# S$ is nilpotent for $j < b(R)$.  However, if $b(R)<\infty$, then, \[ [H^{b(R)}_{\fm_R}(R)\# S]_0 = [H^{b(R)}_{\fm_R}(R)]_0\otimes_k S_0 \cong [H^{b(R)}_{\fm_R}(R)]_0\] is not nilpotent in degree 0 by hypothesis. Consequently $H^{b(R)}_{\fm_T}(T)$ has a non-nilpotent summand. Similarly, if $b(S)<\infty$, then $H^{b(S)}_{\fm_T}(T)$ is not nilpotent. 

For convenience, set $M_j = \bigoplus_{r+s=j+1} H_{\fm_R}^r(R) \# H_{\fm_S}^s(S)$. The first index $j$ for which $M_j$ may contain a non-nilpotent summand is when $j+1=f_R + f_S$ by Lemma~\ref{lem:nilSegre}, i.e. when $j=f$. Thus, for $j<f$, $M_j$ is nilpotent. Consequently, when $j<\min\{f,b(R),b(S)\}$, $H^j_{\fm_T}(T)$ is nilpotent as required. 

When $j=f$, the summand $H^{f_R}_{\fm_R}(R)\# H^{f_S}_{\fm_S}(S)$, which is nilpotent if and only if $H^{f_R}_{\fm_R}(R)$ and $H^{f_S}_{\fm_S}(S)$ have a non-nilpotent degree in common. One way to guarantee this is the stated condition, that $b(R)=f_R$ and $b(S)=f_S$. 
\end{proof}

The conditions given above are only sufficient since computing even one of the nil-supports involved is hard outside of the given cases, see Remark~\ref{rmk:infinite nilsup is fucking annoying}. However, at maximal $F$-depth the conditions are not too restrictive.

\begin{cor}\label{cor:wFnSegre}
If $R$ and $S$ are weakly $F$-nilpotent, then $T$ is weakly $F$-nilpotent if and only if $b(R)=b(S)=\infty$, in which case $b(T)=\infty$ as well.
\end{cor}

Interestingly, for generalized $F$-depth one of the bounds for the ``expected" formula for Segre products holds without additional hypotheses, unlike depth and $F$-depth. 

\begin{thm}\label{thm:gfdepthSegre}
Write $g_U = \gfdp U$ for $U \in \{R,S,T\}$.  We have the lower bound $g_T \ge g_R + g_S -1$, with equality holding if and only if \[\nilsup H^{g_R}_\fm(R)\cap \nilsup H^{g_S}_\fn(S) \not\subset \{0\}.\]In particular, if $R$ and $S$ are generalized weakly $F$-nilpotent, so is $T$. 
\end{thm}

\begin{proof}
Adopt the notation for $M_j$ as used in the proof of Theorem~\ref{thm:fdepthSegre}. The summands $H^j_{\fm_R}(R) \# S$ and $R \# H^j_{\fm_S}(S)$ are generalized nilpotent for all $j$ since their degree supports are finite. By utilizing Lemma~\ref{lem:nilSegre}, we can see that the first index which might contribute a summand which is not generalized nilpotent is when $j+1=g_R+g_S$, which shows the desired inequality. 

If the degree support condition is satisfied, we can see that the intersection is infinite by Lemma~\ref{lem:graded nilp and gen nilp}. Thus, $H^{g_T}_{\fm_T}(T)/0^F_{H^{g_T}_{\fm_T}(T)}\not\subset\{0\}$ since it is an infinite set. So, $H^{g_T}_{\fm_T}(T)$ is not generalized nilpotent.
\end{proof}
\subsection{Examples}
We use the formulas for $F$-depth and generalized $F$-depth to study Segre products of Fermat hypersurfaces and the projective line.	

\begin{xmp}\label{xmp:Kunneth}
Let $R = k[x,y,z]/(x^4+y^4-z^4)$ and $S=k[u,v]$. Both are Cohen-Macaulay rings of dimension $2$, with $a(R)=0$ and $a(S)=-2$. Set $T=R\# S$. Below is a decomposition of the local cohomology for $T$ using the K\"unneth formula. 
\[\begin{array}{ccc}
H^0_{\fm_T}(T) &=& 0 \\
H^1_{\fm_T}(T) &=& 0 \\ 
H^2_{\fm_T}(T) &=& H^2_{\fm_R}(R) \# S \\
H^3_{\fm_T}(T) &=& H^2_{\fm_R}(R) \# H^2_{\fm_S}(S)
\end{array}\]

Since $a(R)=0$, $H^2_{\fm_T}(T)\neq 0$ and thus $\depth(T)=2$, so $T$ is not Cohen-Macaulay as $\dim T=3$, however $T$ is generalized Cohen-Macaulay hence generalized weakly $F$-nilpotent. One may also use Theorem~\ref{thm:gfdepthSegre} to see that $T$ is generalized weakly $F$-nilpotent.

When $p \equiv 1 \bmod 4$, then $b_2(R) = 0$ so $H^2_{\fm_T}(T)$ is not nilpotent and consequently $\fdp T=2$. However, when $p \equiv 3 \bmod 4$, then $H^2_{\fm_T}(T)$ is nilpotent so $\fdp T=3$ and $T$ is weakly $F$-nilpotent.
\end{xmp}

By using Lemma~\ref{lem:FermatHyp}, a similar argument as in the previous example gives the following result.

\begin{thm}\label{thm:FermatSegre}
Let $p>d$ and $n \geq 2$, and let $R = k[x_0,\ldots,x_n]/(x_0^d + \cdots + x_{n-1}^d - x_n^d)$, $S = k[u,v]$, and $T=R\# S$. If $p \equiv 1 \bmod d$, then $b_2(R) = 0$ and $\fdp T = 2$. When $p \equiv -1 \bmod d$, then $H^2_{\fm_T}(T)$ is nilpotent so $\fdp T=3$ and $T$ is weakly $F$-nilpotent.
\end{thm}

\subsection{Frobenius test exponents for Segre products}

To apply the finiteness theorems to Segre products which have nilpotent singularity types, we first explicate the relationship between the Hartshorne-Speiser-Lyubeznik numbers of $M\# N$ in terms of those for $M$ and $N$.

\begin{lem}\label{lem:HSL bounds for Segre Products}
If $M$ is nilpotent then $\hsl M\# N \le \hsl M$ (and similarly if $N$ is nilpotent). Consequently, if $M$ and $N$ are nilpotent, $\hsl M\# N\le\min\{\hsl M,\hsl N\}$.

\end{lem}
\begin{proof}
The upper bounds are vacuous if the right hand side is infinite, so we assume $\hsl M$ and $\hsl N$ are finite. If $\hsl M=e<\infty$, we have $\rho^e: M \rightarrow M$ is zero, which implies $\rho^e: M \# N \rightarrow M\# N$ is also zero, as the action is diagonal.  This argument is clearly symmetric in $M$ and $N$.

When both $M$ and $N$ are nilpotent, if either $\hsl M$ or $\hsl N$ are finite, we can apply the previous case to obtain $\hsl M\# N\le \min \{ \hsl M,\hsl N\}$.
\end{proof}


We now apply this lemma to the summands in the K\"unneth formula.  For rotational convenience, set $$\fexp_j(R)= \min\{e \in \NN \mid p^e > a_j(R)\}.$$ This number is an obvious upper bound for the required iterate of $F$ to uniformly annihilate the non-negative degree part $[H^j_{\fm_R}(R)]_{>0} \subset H^j_{\fm_R}(R)$ and depends only on the graded structure of the ring.

\begin{thm}\label{thm:HSL bounds for Kunneth formula}
If $R$ and $S$ are weakly $F$-nilpotent, then  \[ \hsl H^j_{\fm_T}(T) \le \max\left\lbrace 
\begin{array}{l}
\hsl [H^j_{\fm_R}(R)]_0,\hsl [H^j_{\fm_S}(S)]_0,\fexp_j R,\fexp_j S,\\ \\
\max\{\hsl H^r_{\fm_R}(R),\hsl H^s_{\fm_S}(S) \mid r+s=j+1 \}
\end{array}
\right\rbrace.
\]
\end{thm}
\begin{proof}
We analyze each summand of the Kunneth formula above. We then combine several maxima. 

First, since $S$ is supported in degree $\ge 0$, we have $H^j_{\fm_R}(R)\# S$ is also supported in degree $\ge 0$. If $a_j(R)\le 0$, then $\fexp_j(R)=0$ which we may remove from the max. If $a_j(R)>0$, then $0^F_{H^j_{\fm_R}(R)\# S} = 0^F_{[H^j_{\fm_R}(R)]_0} \oplus Q$ where $Q$ is only possibly supported in degrees $n$ with $0 < n \le a_j(R)$. Consequently, $F^{\fexp_j}(Q)=0$. We then apply Corollary~\ref{cor:finite support hsl bound} and Lemma~\ref{lem:lift of nilp is nilp hsl bounds}. The complementary summand $R\# H^j_{\fm_S}(S)$ is handled identically. The final summand follows by a similar analysis of each term. 

\end{proof}

\begin{rmk}
Note that as we successively compute the data required bound $\hsl H^j_{\fm_T}(T)$, we do not actually need to know $\hsl H^j_{\fm_R}(R)$ or $\hsl H^j_{\fm_S}(S)$, only \textit{a priori} smaller numbers $\hsl [H^j_{\fm_R}(R)]_0$ and $\hsl [H^j_{\fm_S}(S)]_0 \le \hsl H^j_{\fm_S}(S)$, as well as $a_j(R)$ and $a_j(S)$.

We also note that above bound is not optimal. We give this form as it is compact. One can improve it by including the data of the Hartshorne-Speiser-Lyubeznik number for $Q$, since $\fexp_j(R)$ is only a coarse upper bound for $\hsl Q$. We may also use $\hsl R$ and $\hsl S$ as very coarse upper bounds to the terms in the maximum above. \end{rmk}

\begin{cor}\label{cor:Fte for Segre Products} If we set $H_j$ the maximum appearing on the right hand side of the estimate in Theorem~\ref{thm:HSL bounds for Kunneth formula}, then when $R$ and $S$ are weakly $F$-nilpotent and $b(R)=b(S)=\infty$, $\fte^* T \leq \sum_{j=0}^{d_T} {\binom{d}{j}} H_j$. In particular, $$\fte^* T \leq 2^{d_T}\max\{ \hsl R,\hsl S\}.$$
\end{cor}

\begin{rmk}\label{rmk:lengths of non-nilpotent quotients for Segre}
Let $j<\gfdp R + \gfdp S - 1\le \gfdp T$. We may use the ideas in the proof of Theorem~\ref{thm:gfdepthSegre} to understand the lengths of $H^j_{\fm_T}(T)$. For notational convenience, we set \[ G^j(U)=\left[H^j_{\fm_U}(U)\right]_0/\left[0^F_{H^j_{\fm_U}(U)}\right]_0\] for $U \in \{R,S,T\}$. By the K\"unneth formula and the fact that length is additive on short exact sequences, we have \[\lambda_T\left(H^j_{\fm_T}(T)/0^F_{H^j_{\fm_T}(T)}\right) = \lambda_R(G^j(R))+\lambda_S(G^j(S))+\sum_{\substack{r+s=\\j+1}} \lambda_T \left( H^r_{\fm_R}(R) \# H^s_{\fm_S}(S)/0^F_{H^r_{\fm_R}(R)\# H^s_{\fm_S}(S)} \right).\] 

The last collection of summands is most easily understood when $R$ and $S$ are generalized Cohen-Macaulay, so that $T$ is also generalized Cohen-Macaulay. In this case, we see \[ \lambda_T \left( H^r_{\fm_R}(R) \# H^s_{\fm_S}(S)/0^F_{H^r_{\fm_R}(R)\# H^s_{\fm_S}(S)} \right) = \lambda_R(G^r(R))\cdot \lambda_S(G^s(S)) = \dim_k(G^r(R))\cdot \dim_k(G^s(S)).\]
\end{rmk}

\begin{cor}\label{cor:fte* for gwfn Segre}
Let $R$ and $S$ be generalized weakly $F$-nilpotent. Let $N$ be minimal so that for all $0 \le j < d_R$ and $0 \le i < d_S$ \[ \fm_R^N H^j_{\fm_R}(R) \subset 0^F_{H^j_{\fm_R}(R)} \hspace{5pt} \text{ and } \hspace{5pt} \fm_S^N H^i_{\fm_S}(S)\subset 0^F_{H^i_{\fm_S}(S)}.\] Further, let $e_1$  be minimal so that $p^{e_1} \ge (N+1)2^{d_T-1}$. We have $\fte^* T \le e_1+\sum_{j=0}^{d_T} \binom{d_T}{j} H_j$, where $H_j$ are the bounds on the Hartshorne-Speiser-Lyubeznik numbers of $T$ given in Theorem~\ref{thm:HSL bounds for Kunneth formula}. 
\end{cor}

\begin{proof}
Note first that $\fm_R [H^j_{\fm_R}(R)]_0$ is nilpotent for all $j$, including when $j=d_R$ as $H^j_{\fm_R}(R)$ vanishes in high degree, and similarly for $S$. Consequently, $\fm_T (H^j_{\fm_R}(R) \# S)$ and $\fm_T(R \# H^j_{\fm_S}(S))$ are nilpotent for all $j$. It may be the case that $T$ is weakly $F$-nilpotent so that the $e_1$ may seem superfluous. The role of the $N+1$ factor in the statement ensures $H^{d_R}_{\fm_R}(R)\# S$ and $R\#H^{d_S}_{\fm_S}(S)$ are sent into the nilpotent part of $H^{d_R}_{\fm_T}(T)$ and $H^{d_S}_{\fm_T}(T)$ respectively. 

Now we need only show $\fm^N_T(H^r_{\fm_R}(R)\# H^s_{\fm_S}(S))$ is nilpotent  for $0 \le j < d_T$ and $r+s=j+1$. However, since $R,S,$ and $T$ are all standard graded, if $x\# y$ is a homogeneous element in $\fm_T^N$, we must have that $x \in \fm_R^N$ and $y\in \fm_S^N$. Consequently, for $r<d_R$ and $s<d_S$, multiplication by $x\# y$ sends $H^r_{\fm_R}(R)\# H^s_{\fm_S}(S)$ into $0^F_{H^j_{\fm_T}(T)}$.
\end{proof}

\section{Veronese Subrings} \label{sec:veronese}

Let $R$ be a graded ring and let $M$ be a $\ZZ$-graded $R$-module. For $n \in \NN$, the \textbf{$n$th Veronese subring of $R$} is the graded subring defined by \[R^{(n)} = \oplus_{t\in\NN} [R]_{nt}.\] Indeed, $R^{(n)}$ is a graded subring of $R$ with $[R^{(n)}]_0=k$ and homogeneous maximal ideal $\fm \cap R^{(n)}=\fm^{(n)} = \oplus_{t \ge 1} [R]_{nt}$. Similarly, the \textbf{$n$th Veronese submodule} $M^{(n)}$ of $M$ is the $\ZZ$-graded $R^{(n)}$-module \[
M^{(n)} = \bigoplus_{t \in \ZZ} [M]_{nt},
\] and in the case $M=R$, $R^{(n)}$ is called the \textbf{$n$th Veronese subring} of $R$. 
For any $n \in \NN$, it is clear that $M^{(n)}$ is a direct summand of $M$ as an $R^{(n)}$-module, in particular, $R^{(n)}$ is a direct summand of $R$ as an $R^{(n)}$-module. Under the inclusion $R^{(n)}\subset R$, $\fm\cap R^{(n)} = \fm^{(n)}$, so we see immediately that $\dim R^{(n)} = \dim R$ for any $n\in \NN$. 
\begin{rmk}\label{rmk:natural versus standard grading}
We can view $M^{(n)}$ with two different gradings. We call the \textbf{natural grading} the one under which $M^{(n)}\subset M$ is homogeneous of degree $0$, and the \textbf{standard grading} after we have regraded, i.e. $[M^{(n)}]_t = [M]_{nt}$. We will typically use the standard grading, but write $\degsup^{\text{nat}}$ when considering the natural grading. Notice $\degsup^{\text{nat}} M^{(n)} = \degsup M \cap n\ZZ$ and $\degsup M^{(n)} = f\left(\degsup^{\text{nat}} M\right) \cap \ZZ$, where $f\colon \ZZ\rightarrow \QQ$ is the function $t\mapsto t/n$. 
\end{rmk}

\begin{run}
We let $R$ be a standard graded ring with homogeneous maximal ideal $\fm$, and fix an $n \in \NN$. For convenience, write $S=R^{(n)}$ with homogeneous maximal ideal $\fn=\fm^{(n)}$.
\end{run}

\subsection{Nilpotent singularity types for Veronese subrings}

\begin{lem}
Suppose $M$ is a graded $R[F]$-module such that $[M]_t$ is a finite dimensional $k$-vector space for all $t \in \ZZ$, and write $N=M^{(n)}$. As an $S$-module, $N$ enjoys a graded $S[F]$-module structure and further, $\nilsup^{\text{nat}} N = \nilsup M \cap n\ZZ$. Thus, if $M$ is (generalized) nilpotent, so is $N$, and $\hsl N\le\hsl M$.
\end{lem}

\begin{proof}
It suffices to show that $N$ is a $\rho$-stable $S$-submodule of $M$, and in particular, that $\rho(m)\in N$ for all homogeneous $m\in N$. This will also prove the required Hartshorne-Speiser-Lyubeznik bounds. If $m \in [N]_{nt} = [M]_{nt}$, then $\rho(m)\in [M]_{ntp} =  [N]_{tp}$. Further, the restricted action $\rho: N \rightarrow N$ is $p$-linear over $S$ since $\rho:M\rightarrow M$ is $p$-linear over $R$, showing the first claim.

Now, by Remark~\ref{rmk:forgetful fctr} $M$ is nilpotent as an $R[F]$-module if and only if it is nilpotent as an $S[F]$-module. The fact that $[N]_{nt}=[M]_{nt}$ shows that $nt\in\nilsup N$ if and only if $nt\in\nilsup M$. Since $\degsup^{\text{nat}} N\subset n\ZZ$, $\nilsup N= n\ZZ\cap \nilsup M$, as claimed. 
\end{proof}

We now show descent of nilpotent singularity types for Veronese subrings. 

\begin{thm}\label{thm:F-depth and gF-depth for Veronese}
We have $f:=\fdp R \le \fdp S$, and if $b_f(R)=0$, then $\fdp S=f$. Furthermore, $\gfdp R\le\gfdp S$. In particular, if $R$ is (generalized) weakly $F$-nilpotent, so is $S$. 
\end{thm}

\begin{proof}
For any $t \in \ZZ$ we have the following commutative diagram of $k$-vector spaces where the horizontal maps are isomorphisms and the vertical maps the restriction of the canonical Frobenius action on local cohomology to a single degree. \begin{center}
\begin{tikzcd}
\left[H^j_{\fn}(S)\right]_t \arrow{r}{\sim} \arrow{d}{F} & \left[H^j_{\fm_R}(R)\right]_{nt} \arrow{d}{F} \\
\left[H^j_{\fn}(S)\right]_{pt} \arrow{r}{\sim}  & \left[H^j_{\fm_R}(R)\right]_{pnt} 
\end{tikzcd}
\end{center} Consequently, if $H^j_{\fm_R}(R)$ is nilpotent, there is an $e \in \NN$ such that $F^e:H^j_{\fm_R}(R) \rightarrow H^j_{\fm_R}(R)$ vanishes, and so $F^e:[H^j_{\fm_R}(R)]_{nt}\rightarrow [H^j_{\fm_R}(R)]_{p^ent}$ vanishes. This implies $H^j_{\fn}(S)$ is nilpotent. Thus, for $j<f$, $H^j_{\fn}(S)$ is nilpotent, showing $\fdp S\ge f$. Furthermore, as $[H^j_\fn(S)]_0 = [H^j_{\fm_R}(R)]_0$, if $b_f(R) = 0$ then $b_f(S)=0$ as well, implying $\fdp S\le f$, so $\fdp S=f$.

By Lemma~\ref{lem:graded nilp and gen nilp}, we have $H^j_{\fm_R}(R)$ and $H^j_\fn(S)$ are generalized nilpotent if and only if they are nilpotent in all nonzero degrees, and $b_j(R)=0$ if and only if $b_j(S)=0$ since the degree zero part of the local cohomology modules of $R$ and $S$ agree. We may then apply the argument with the commutative diagram above to see that if $H^j_\fm(R)$ is nilpotent in nonzero degree, the same is true of $H^j_\fn(S)$. Thus, $\gfdp R \le \gfdp S$.
\end{proof}

Now apply our results so far to verify when a graded rings with controlled singularities on the punctured spectrum have high Veronese subrings which are $F$-nilpotent, similar to \cite[Prop. 3.1]{Sin00}. Since we vary the degree of the Veronese subring, we will write $S_n=R^{(n)}$ and $\fm_n=\fm \cap S_n$  in the theorem that follows; notably $S_1=R$. 

\begin{thm}\label{thm:VeronesePuncSpec} If $R$ is weakly $F$-nilpotent and $F$-nilpotent on the punctured spectrum, then the following are equivalent.
\begin{enumerate}[label=(\alph*)]
\item $R$ is $F$-nilpotent
\item $b(R)=\infty$
\item for all $n \in\NN$, $S_n$ is $F$-nilpotent
\item for a single $n \in \NN$, $S_n$ is $F$-nilpotent
\end{enumerate}
\end{thm}
\begin{proof} 
The equivalence of (a) and (b) is demonstrated in Lemma~\ref{lem:punctured spec}. Let $n \in \NN$ and for convenience write \[H(n)=0^*_{H^d_{\fm_n}(S_n)}/0^F_{H^d_{\fm_n}(S_n)}\] and $H=H(1)$. Since $R$ is $F$-nilpotent on the punctured spectrum, $H$ is finite length, so $\degsup H \subset \{0\}$. Further, the graded inclusion $S_n \subset R$ implies that $\degsup H(n)\subset \degsup H$. But since $H^d_{\fm_n}(S_n)\simeq H^d_\fm(R)^{(n)}$ for all $n \in \NN$, we have $[H(n)]_0=[H]_0$ for all $n \in \NN$ as $k[F]$-modules. Hence, $\degsup H(n) = \degsup H$ for all $n$. This shows the remainder of the equivalences.
\end{proof}

\subsection{Frobenius test exponents for Veronese subrings}

We can now bound the homogeneous Frobenius test exponent for a Veronese subring of a graded ring with a nilpotent singularity. 

\begin{thm}\label{thm:Fte for Veronese}
One has $\hsl H^j_{\fm_{S}}(R') \le \hsl H^j_{\fm_R}(R)$. Consequently, if $R$ is weakly $F$-nilpotent, then $\fte^* S \le \sum_{j=0}^d \binom{d}{j} \hsl H^j_{\fm_R}(R)$. Furthermore, if $R$ is generalized weakly $F$-nilpotent and $N$ is the smallest $m \in \NN$ such that $\fm_R^m H^j_{\fm_R}(R)$ is nilpotent for all $0 \le j <d$ and $e_1$ is the smallest $e \in \NN$ such that $p^e \ge N 2^{d-1}$, then $\fte^* S \le e_1+\sum_{j=0}^d \binom{d}{j} \hsl H^j_{\fm_R}(R)$.
\end{thm}

\begin{proof}
We recall the diagram from the proof of Theorem~\ref{thm:F-depth and gF-depth for Veronese}, from which the first claim follows immediately, \begin{center}
\begin{tikzcd}
\left[ H^j_{\fm_{R^{(n)}}}(R^{(n)})\right]_t \arrow{r}{\sim} \arrow{d}{F} & \left[H^j_{\fm_R}(R)\right]_{nt} \arrow{d}{F} \\
\left[H^j_{\fm_{R^{(n)}}}(R^{(n)})\right]_{pt} \arrow{r}{\sim}  & \left[H^j_{\fm_R}(R)\right]_{pnt} 
\end{tikzcd}
\end{center} 

Now, given an $x \in \fm_{S}^N$, viewed in $R$ we also have $x \in \fm_R^N$. Since $x H^j_{\fm_R}(R)$ is nilpotent, $xH^j_{\fm_R}(R)$ is nilpotent in degree $nt$ for all $t \in \ZZ$, completing the claim.
\end{proof}

\section{Diagonal subalgebras of bigraded hypersurfaces} 

In Sections~\ref{sec:segre} and \ref{sec:veronese}, we developed the necessary tools to understand nilpotence for Segre products and Veronese submodules. Veronese subalgebras of multigraded hypersurface rings provide interesting examples of $\NN$-graded rings with distinguished $F$-singularities which were studied in \cite{KSSW09}. Their local cohomology admits a K\"unneth-type decomposition into Segre products and Veroneses, see \cite[Lem. 2.1]{KSSW09}, so our tools will apply to understanding when they have nilpotent singularity types. We review the necessary details of their construction below.

Let $T$ be an $\NN^2$-graded ring in the sense of \cite{GW78b}, where $[T]_{(0,0)}=k$ is an algebraically closed field and $T$ is a domain. Given a $\ZZ^2$-graded $T$-module $M$, we have corresponding notions of degree shifts and Veronese submodules as outlined below.

\begin{dff}
Let $M$ be a $\ZZ^2$-graded $T$ module, and fix $(d,e)\in \ZZ^2$ and $\Delta=(g,h)\in\NN^2$. \begin{itemize}
\item The \textbf{graded shift} $M(d,e)$ of $M$ is a new graded $T$-module structure on $M$ by $[M(d,e)]_{(a,b)}=[M]_{a+d,b+e}$. 
\item The \textbf{Veronese subring} of $T$ defined by $\Delta$ is the ring \[T_\Delta=
\bigoplus_{t \in \NN} [T]_{(gt,ht)} 
\] which can be viewed naturally as an $\NN^2$-graded subring of $T$ for which the graded inclusion $T_\Delta\subset T$ is split as $T_\Delta$-modules, or alternatively as an $\NN$-graded ring by $[T_\Delta]_t = [T]_{(gt,ht)}$. 
\item Finally, the \textbf{Veronese submodule} of $M$ defined by $\Delta$ is given by \[
[M_\Delta]_t = \bigoplus_{t \in \ZZ} [M]_{(gt,ht)}
\] which is also naturally viewed a $\ZZ^2$-graded direct summand of $M$ as a $T_\Delta$-module or a $\ZZ$-graded $T_\Delta$-module.
\end{itemize}
\end{dff} 

\begin{rmk}\label{rmk:diagonal functor}It is convenient to view $\bullet_\Delta$ as a functor which sends objects to their ``diagonal" defined by $\Delta$ and the maps to their restrictions to this diagonal. This makes sense either as a functor from the category of $\ZZ^2$-graded abelian groups to the category of $\ZZ$-graded abelian groups or the category of $\ZZ^2$-graded $T$-modules to $\ZZ$-graded $T_\Delta$-modules. From this, it is clear that if $M$ has a $\ZZ^2$-graded $T[F]$-module structure\footnote{We do not spell this idea out carefully here, but the extensions to the bigraded case are routine.} $(M,\rho)$, $M_\Delta$ has a $\ZZ$-graded $T_\Delta[F]$-module structure $(M_\Delta,\rho_\Delta)$. Further, if $(M,\rho)$ is nilpotent, so is $(M_\Delta,\rho_\Delta)$. By using the $\NN$-grading on $T_\Delta$, we can apply all of the general theory developed earlier in the paper for graded $T_\Delta[F]$-modules.\end{rmk}

When $T$ is a tensor product of polynomial rings $A$ and $B$ of dimensions $m$ and $n$ respectively, the K\"unneth formula gives simple criteria for when $(T/f)_\Delta$ is Gorenstein or Cohen-Macaulay. In particular, \cite[Thm.~3.1]{KSSW09} identifies that $(T/f)_\Delta$ is Cohen-Macaulay if and only if $\lfloor \frac{d-m}{g} \rfloor < \frac{e}{h}$ and $\lfloor \frac{e - n}{h} \rfloor < \frac{d}{g}$. 

\begin{run}
Let $(R,\fm)$ and $(S,\fn)$ be standard graded normal domains of dimension $d_R$ and $d_S$ respectively, where $[R]_0=[S]_0=k$ is an algebraically closed field. We write $f_R=\fdp R$ and $f_S=\fdp S$. 

Further, we assume $R$ and $S$ are finite type over $k$, so that $R$ is generated as a $k$-algebra in degree $1$ by some elements $x_1,\cdots,x_r$ and $S$ is generated as a $k$-algebra in degree $1$ by elements $y_1,\cdots,y_s$. Define an $\NN^2$-grading on $T=R\otimes_k S$, notably a normal domain, by setting $\deg(x_i\otimes 1)=(1,0)$ and $\deg(1\otimes y_i)=(0,1)$, and extend the grading to the whole ring.

\

Finally, we fix a nonzero $f \in T$ of degree $\deg(f)=(d,e)>(0,0)$ and suppose $\Delta=(g,h)\in \NN^2$. We write $\fM$ for the homogeneous maximal ideal of $T_\Delta$.
\end{run}

\begin{rmk}\label{rmk:diag description}
In this setting, we can see $T_\Delta =R^{(g)}\# S^{(h)}$, where we have regraded $R^{(g)}$ and $S^{(h)}$ to be standard graded. On can see $T_\Delta$ as a simultaneous generalization of the earlier constructions in the following sense for particular choices of $\Delta$.
\begin{itemize}
\item $\Delta=(1,0)$: Since $S^{(0)}$ is simply $k$ in degree 0 and $R^{(1)}=R$, we have $$T_\Delta = k\otimes_k k \oplus \bigoplus_{t\ge 1} [R]_t = \oplus_{t\in\NN} [R]_t = R.$$ In this case, we see $\dim T_\Delta = d_R$.
\item $\Delta_1=(g,0)$ or $\Delta_2=(0,h)$: Similarly, we have $T_{\Delta_1} = R^{(g)}$ with $\dim T_{\Delta_1} = d_R$, and by symmetry, $T_{\Delta_2} = S^{(h)}$ with $\dim T_{\Delta_2} = d_S$.
\item $\Delta =(1,1)$: We recover the Segre product, $T_\Delta = R\# S$, and $\dim T_\Delta = d_R + d_S - 1$
\end{itemize}
We can now see that when $g\ge 1$ and $h \ge 1$, $\dim T_\Delta = \dim R^{(g)} + \dim S^{(h)} - 1=  d_R + d_S - 1$. From now on, we will assume $(g,h)>(0,0)$. 
\end{rmk}

We are now ready to apply our theorems on Segre products and Veronese subrings to study the (generalized) $F$-depth of $T_\Delta$. Many of the conditions are only sufficient because $\nilsup H^{f_R}_\fm(R)$ and $\nilsup H^{f_S}_\fn(S)$ are both difficult to understand in general  

\begin{thm}\label{thm:f-depth of diagonal subalgebras}
For convenience, set $f_T=f_R+f_S-1$. We have the following,
\begin{enumerate}[label=(\alph*)]
\item $\fdp T_\Delta \ge f_T$,
\item $H^{f_T}_{\fM}(T_\Delta)$ is generalized nilpotent if either $H^{f_R}_\fm(R)$ or $H^{f_S}_\fn(S)$ is generalized nilpotent,
\item $\fdp T_\Delta = f_T$ if $b(R)=f_R$ and $b(S)=f_S$,
\item $\fdp T_\Delta > f_T$ if $b_{f_T}(R)\neq 0$, $b_{f_T}(S)\neq 0$, and either $H^{f_R}_\fm(R)$ is generalized nilpotent and $b_{f_S}(S)\neq 0$ or $H^{f_S}_\fn(S)$ is generalized nilpotent and $b_{f_R}(R)\neq 0$.
\end{enumerate}
\end{thm}
\begin{proof}
Since $T_\Delta = R^{(g)}\# S^{(h)}$, we can use Theorems~\ref{thm:fdepthSegre} and \ref{thm:F-depth and gF-depth for Veronese} to see: \[
\fdp T_\Delta \ge \fdp R^{(g)}+\fdp S^{(h)} - 1 \ge f_R + f_S-1=f_T
\]  To see the remainder of the claims, we use the K\"unneth formula for $T_\Delta=R^{(g)}\# S^{(h)}$ and analyze each summand.\[\renewcommand{\arraystretch}{1.25}
\begin{array}{l|l}
H^{f_T}_\fm\left(R^{(g)}\right) \# S^{(h)} & \text{gen. nilpotent, nilpotent if and only if } b_{f_T}(R)\neq 0 \\
R^{(g)}\# H^{f_T}_\fn\left(S^{(h)}\right) & \text{gen. nilpotent, nilpotent if and only if } b_{f_T}(S)\neq 0 \\
H^i_\fm\left(R^{(g)}\right)\# H^j_\fn\left(S^{(h)}\right) \text{ for } (i,j)<(f_R,f_S) & \text{nilpotent} \\
H^{f_R}_\fm\left(R^{(g)}\right) \# H^{f_S}_\fn\left(S^{(h)}\right) & \text{described below} 
\end{array}\]

Now consider $H=H^{f_R}_\fm\left(R^{(g)}\right) \# H^{f_S}_\fn\left(S^{(h)}\right)$. If $H^{f_R}_\fm(R^{(g)})$ is generalized nilpotent but $b_{f_S}(S)<0$, we see that $H$ is nilpotent, with the obvious symmetric statement swapping the roles of $R$ and $S$. If both are generalized nilpotent, then $H$ is is non-nilpotent in degree zero. This completes the remainder of the claims.
\end{proof}

To determine when $(T/fT)_\Delta$ also has nilpotent singularity types, we need to carefully understand the interrelated $T_\Delta[F]$-modules involved.

\begin{lem}\label{lem:diag hyp diagram}
We then have a short exact sequence of graded $T_\Delta[F]$-modules as below. \begin{center}
\begin{tikzcd}
0 \arrow{r} & (T(-d,-e)_\Delta ,(f^{p-1}F)_\Delta) \arrow{r}{(\cdot f)_\Delta} & (T_\Delta ,F) \arrow{r} &  ((T/fT)_\Delta ,F) \arrow{r} & 0 
\end{tikzcd}
\end{center} Writing $(T',\rho)$ for $(T(-d,-e)_\Delta,f^{p-1}F)_\Delta)$, we have the inequalities below. \[\fdp T_\Delta \le \fdp(T',\rho)\text{ and }\gfdp T_\Delta\le\gfdp(T',\rho)\]
\end{lem}

\begin{proof}
Notably, $T/fT$ is another $\NN^2$-graded ring with grading given by $\deg(g+fT)=\deg(g)$ for a homogeneous $g+fT\in T/fT$. We have the following commutative diagram of graded $T$-modules with exact rows, whose horizontal maps are degree $(0,0)$ and vertical maps are graded $p$-linear.\begin{center}
\begin{tikzcd}
0 \arrow{r} & T(-d,-e)\arrow{d}{f^{p-1}F} \arrow{r}{\cdot f} & T\arrow{d}{F} \arrow{r} & T/fT\arrow{d}{F} \arrow{r} & 0 \\
0 \arrow{r} & T(-d,-e) \arrow{r}{\cdot f} & T\arrow{r} & T/fT \arrow{r} & 0 
\end{tikzcd}
\end{center} The map $f^{p-1}F:T(-d,-e)\rightarrow T(-d,-e)$ deserves some discussion. \begin{itemize}
\item For any homogeneous $x \in [T(-d,-e)]_{(a,b)}=[T]_{(a-d,b-e)}$, we have $F(x)\in[T]_{(ap-dp,bp-ep)}$, and $\deg(f^{p-1}) = (dp-d,ep-e)$, so $f^{p-1}F(x) \in [T]_{(ap-d,bp-e)} = [T(-d,-e)]_{(ap,bp)}$. Further, for any $z\in T$, $f^{p-1}F(zx)=z^p(f^{p-1}F(x))$. If we write $\rho=f^{p-1}F$, we can see that $(T(-d,-e),\rho)$ is a graded $T[F]$-module. 
\item There is a simple formula for the iterates of $\rho$, namely for any $w$, $\rho^w = f^{p^w-1}F^w$. 
\end{itemize}
The rest follows from Remark~\ref{rmk:diagonal functor}, observing that the canonical Frobenius map on an $\NN^2$-graded ring passes to the canonical Frobenius map on the corresponding $\NN$-graded ring under $\bullet_\Delta$. The fact that $F_\Delta=F$ and the second bullet point above together imply the $F$-depth and generalized $F$-depth inequalities. \end{proof}

We can now explicate some situations in which $(T/fT)_\Delta$ has nilpotent-type singularities. 

\begin{thm}\label{thm:diagonalSubalgebraHypersurface} We have the following inequalities,
\begin{itemize}
\item $\fdp (T/fT)_\Delta \ge \fdp T_\Delta -1$,
\item $\gfdp (T/fT)_\Delta \ge \gfdp T_\Delta -1$.
\end{itemize} Thus, if $\dim (T/fT)_\Delta = \dim T_\Delta - 1$ and $T_\Delta$ is (generalized) weakly $F$-nilpotent, so is $(T/fT)_\Delta$.
\end{thm}

\begin{proof}
We use the graded version of Theorem~\ref{thm:Fdepthses} applied to the short exact sequence in \ref{lem:diag hyp diagram}, from which the inequalities follow. 
\end{proof}

On the other hand, $(T/fT)_\Delta$ may be weakly $F$-nilpotent even if $T_\Delta$ is not. We give an example of a sufficient condition below.

\begin{xmp}
Suppose $\dim (T/fT)_\Delta = \dim T_\Delta - 1$, and $H^{f_R}_\fm(R)$ and $H^{f_S}_\fn(S)$ are generalized nilpotent so $\fdp T_\Delta = f_T=f_R+f_S-1$ by Theorem~\ref{thm:f-depth of diagonal subalgebras}(c) and Lemma~\ref{lem:graded nilp and gen nilp}. If $\fdp T_\Delta < \dim T_\Delta$ and  the following conditions are satisfied,  
\begin{itemize}
\item $\displaystyle \dfrac{e}{h} \ge \left\lceil \dfrac{b_{f_T}(R)+d}{g} \right\rceil$ or $\dfrac{e}{h}\not\in \ZZ$,
\item $\displaystyle \dfrac{d}{g} \ge \left\lceil \dfrac{b_{f_T}(S)+e}{h} \right\rceil$ or $\dfrac{d}{g}\not\in \ZZ$,
\item the integer matrix $\left[ 
\begin{array}{cc}
d & e \\
g & h
\end{array}
\right]$ has nonzero determinant or $\left\lbrace \dfrac{e}{h},\dfrac{d}{g}\right\rbrace\not\subset \ZZ$, 
\end{itemize} then $\fdp (T/fT)_\Delta \ge \fdp T_\Delta$. To see this, by Lemma~\ref{lem:diag hyp diagram}, note suffices to show $\fdp(T',\rho)>f_T$, which we can do by showing that the shifted Frobenius map $F^w:H^{f_T}_\fM(T') \rightarrow H^{f_T}_\fM(T(-dp^w,-ep^w)_\Delta)$ vanishes for $w\gg 0$ by Remark~\ref{rmk:forgetful fctr}. 

Define $\mcl_1(x) = -d+gx$ and $\mcl_2(x) = -e+hx$ viewed as linear functions $\RR\rightarrow \RR$, notably bijective. By \cite[Lemma~2.1]{KSSW09},  a modified version of the K\"unneth formula gives that $H^{f_T}_\fM(T') = A_R \oplus A_S \oplus H \oplus N$, where the summands are given by,
\begin{align*}
[A_R]_t =& \left[H^{f_T}_\fm(R)\right]_{\mcl_1(t)} \# [S]_{\mcl_2(t)}, \\
[A_S]_t =& [R]_{\mcl_1(t)} \# \left[H^{f_T}_\fn(S)\right]_{\mcl_2(t)}, \\
[H]_t =& \left[H^{f_R}_\fm(R)\right]_{\mcl_1(t)} \# \left[H^{f_S}_\fn(S)\right]_{\mcl_2(t)}, \\
[N]_t =& \bigoplus_{\substack{i+j=f_T+1 \\ i<f_R \text{ and } j<f_S}}\left[H^{i}_\fm(R)\right]_{\mcl_1(t)} \# \left[H^j_\fn(S)\right]_{\mcl_2(t)}.
\end{align*}

From here, we have the following description of the nil-supports.
\begin{itemize}
\item One has $t \in \nilsup A_R$ if and only if $t\in \mcl_1^{-1}\{ \nilsup H^{f_T}_\fm(R)\} \cap \mcl_2^{-1}[0,\infty) \cap \ZZ$. We have \[\mcl^{-1}_1\{\nilsup H^{f_T}_\fm(R)\}\subset \left(-\infty,\dfrac{b_{f_T}(R)+d}{g} \right] \text{ and } \mcl_2^{-1}[0,\infty)\subset \left[\dfrac{e}{h},\infty \right)\] with an obvious symmetric statement for $\nilsup A_S$.
\item Further, one has $t \in \nilsup H$ if and only if $t\in \mcl_1^{-1}\{0\}\cap\mcl_2^{-1}\{0\}\cap \ZZ$. We have $\mcl_1^{-1}(0)=d/g$ and $\mcl_2^{-1}(0)=e/h$ so this intersection is nonempty if and only if $e/h=d/g\in\ZZ$.
\item $\nilsup N=\varnothing$ since either $H^i_\fm(R)$ or $H^j_\fn(S)$ are nilpotent for $i<f_R$ and $j<f_S$.
\end{itemize}
From the analysis of the nil-supports above, the conditions are sufficient to ensure $H^{f_T}_\fM(T')$ is nilpotent under $\rho$.
\end{xmp}

For a concrete example, we present the following.

\begin{rmk}
Note that the theorem above and examples generated thereby can be viewed as similar to counter-examples to the deformation of $F$-nilpotent singularities.  Srinivas-Takagi showed that $F$-nilpotence does not deform even in the Gorenstein case, see \cite[Example~2.7]{ST15}. Polstra-Quy also noted that weak $F$-nilpotence does not deform by showing that there are difficulties at the $H^0$ level, see \cite[Section~4]{PQ19}. 
\end{rmk}

\begin{xmp}
Suppose $k$ is an algebraically closed field of characteristic $7$. Let  \[R=k[x_0,x_1,x_2]/(x_0^4+x_1^4-x_2^4)\,\text{ and }\,S=k[y_0,y_1,y_2]/(y_0^3+y_1^3-y_2^3),\] notably both Cohen-Macaulay normal domains of dimension 2. By Lemma~\ref{lem:FermatHyp} we see that $b(R)=\infty$ and by \cite[Lemma~5.27]{Bli01}, $b(S)=2$. We also have $a(R)=1$ and $a(S)=0$.

For all $\Delta'=(g,h) \in \NN^2$ with $g>1$, we have $a\left(R^{(g)}\right)=a\left(S^{(h)}\right) = 0$, so $\dim T_{\Delta'} = 3>\fdp T_{\Delta'} = 2$, so that $T_{\Delta'}$ is not weakly $F$-nilpotent. Thus, if we take $\Delta=(2,2)$ and $f=x_0\otimes y_0 \in T$ homogeneous of multidegree $(1,1)$, then the numerical conditions in the previous theorem are satisfied, and hence $(T/fT)_\Delta$ is weakly $F$-nilpotent. 
\end{xmp}

\end{document}